\documentclass[12pt, reqno]{amsart}
\usepackage{mathrsfs}
\usepackage{amsfonts}
\usepackage[centertags]{amsmath}
\usepackage{amssymb,array}
\usepackage{amsthm}
\usepackage{stmaryrd}
\usepackage{graphicx}
\usepackage{caption}
\usepackage{ytableau}
\usepackage{enumerate,enumitem}
\usepackage{MnSymbol}
\SetLabelAlign{center}{\hfill #1\hfill}
\usepackage[textwidth=16cm, hmarginratio=1:1]{geometry}
\usepackage{tikz-cd, tikz}
\usepackage{rotating}
\usepackage[all,cmtip]{xy}
\usepackage[titletoc, title]{appendix}
\usepackage{amscd}
\usepackage[colorlinks]{hyperref}
\usepackage{float}
\usepackage{dsfont} 
\hypersetup{
bookmarksnumbered,
pdfstartview={FitH},
breaklinks=true,
linkcolor=blue,
urlcolor=blue,
citecolor=blue,
bookmarksdepth=2
}

\theoremstyle{plain}
\newtheorem{theorem}{Theorem}[section]
\newtheorem{proposition}[theorem]{Proposition}
\newtheorem{lemma}[theorem]{Lemma}
\newtheorem{corollary}[theorem]{Corollary}
\newtheorem{conjecture}[theorem]{Conjecture}

\theoremstyle{definition}
\newtheorem{definition}[theorem]{Definition}
\newtheorem{example}[theorem]{Example}

\theoremstyle{remark}
\newtheorem{remark}[theorem]{Remark}

\numberwithin{equation}{section}

\newcommand{\be}{\begin{enumerate}}

    \newcommand{\ene}{\end{enumerate}}

    \newcommand{\ZZ}{\mathbb{Z}}

    \newcommand{\Id}{\operatorname{Id}}

    \newcommand{\BB}{\operatorname{\mathbf{B}}}
    
    \newcommand{\fn}{\operatorname{\mathfrak{n}}}

    \newcommand{\bb}{\mathbf{b}}
    \newcommand{\bfa}{\mathbf{a}}

    \newcommand{\bfe}{\mathbf{e}}
    \newcommand{\bfi}{\mathbf{i}}
     
    \newcommand{\bfj}{\mathbf{j}}
    \newcommand{\bfk}{\mathbf{k}}

    \newcommand{\bt}{\mathbf{t}}

    \newcommand{\cL}{\mathcal L}

    \newcommand{\cA}{\mathcal A}

\DeclareMathOperator*{\wt}{wt}

\DeclareMathOperator{\Br}{Br}

\newlength{\mysizetiny}
\newlength{\mysizesmall}
\newlength{\mysize}
\newlength{\mysizelarge}
\begin{document}

\title{On cluster structures of bosonic extensions}
\author{Yingjin Bi}
\address{Department of Mathematics, Harbin Engineering University}
\email{yingjinbi@mail.bnu.edu.cn}
\date{} 

\begin{abstract}
We study quantum cluster structures on bosonic extensions of quantum
unipotent coordinate rings. For a positive braid group element
$b\in \operatorname{Br}^+$, Kashiwara--Kim--Oh--Park introduced a
subalgebra $\widehat{\mathcal A}(b)$ and conjectured that it admits a
quantum cluster algebra structure whose cluster monomials belong to the
global basis.

In this paper, we analyze Lusztig parametrizations of the global basis of
$\widehat{\mathcal A}(b)$ and study their transition maps under braid
moves. We prove that the resulting quantum cluster structure is
independent of the chosen expression of $b$.  Combining these ingredients, we prove the Kashiwara--Kim--Oh--Park
conjecture for every \(b\in\operatorname{Br}^+\) in type ADE. Our proof is based on the compatibility between Lusztig parametrizations,
braid moves, and cluster mutations, and is different from the approaches
of Qin and of Kashiwara--Kim--Oh--Park. We also establish quantum
\(T\)-system relations for generalized quantum minors and show that these
minors occur as cluster variables.
\end{abstract}
\maketitle 

\section{Introduction}

Let \(C=(c_{ij})_{i,j\in I}\) be a symmetrizable generalized Cartan
matrix, and let \(A_q(\mathfrak n)\) be the quantum unipotent coordinate
ring associated with \(C\). The bosonic extension
\(\widehat{\mathcal A}\) of \(A_q(\mathfrak n)\) was introduced in finite
type by Oh--Park~\cite{oh2025pbw}, and later generalized to arbitrary
symmetrizable Cartan matrices by Kashiwara--Kim--Oh--Park
\cite{kashiwara2024global}. 

In simply-laced Dynkin type, the bosonic extension \(\widehat{\mathcal A}\)
is closely related to the quantum Grothendieck ring of the
Hernandez--Leclerc category \(\mathscr C_{\mathbb Z}\) of
finite-dimensional representations of the quantum loop algebra
\(U_q(L\mathfrak g)\), where \(\mathfrak g\) is a finite-dimensional simple
Lie algebra~\cite{hernandez2015quantum}. For an interval
\(-\infty\leq a\leq c<+\infty\), let
\(\mathscr C^{[a,c]}\subset \mathscr C_{\mathbb Z}\) denote the
corresponding Hernandez--Leclerc subcategory. The quantum Grothendieck
rings \(K_t(\mathscr C^{[a,c]})\) are known to admit quantum cluster
algebra structures~\cite{hernandez2016cluster,fujita2023isomorphisms}.
It is therefore natural to ask how these cluster structures are related
to the subalgebras \(\widehat{\mathcal A}(b)\) associated with positive
braid group elements.

In~\cite{kashiwara2024braid}, Kashiwara--Kim--Oh--Park constructed braid
group symmetries
\[
T_i:\widehat{\mathcal A}\longrightarrow \widehat{\mathcal A}
\qquad (i\in I),
\]
and proved that these operators satisfy the braid relations. Hence every
positive braid group element \(b\in\Br^+\) defines an algebra automorphism
\[
T_b:\widehat{\mathcal A}\longrightarrow \widehat{\mathcal A}.
\]
The bosonic extension is naturally \(\mathbb Z\)-graded. Let
\(\widehat{\mathcal A}_{\geq0}\) and \(\widehat{\mathcal A}_{<0}\) be the
subalgebras generated by homogeneous elements of nonnegative and negative
degree, respectively. For \(b\in\Br^+\), define
\[
\widehat{\mathcal A}(b)
:=
\widehat{\mathcal A}_{\geq0}
\cap
T_b(\widehat{\mathcal A}_{<0}).
\]
Kashiwara--Kim--Oh--Park proposed the following conjecture.

\begin{conjecture}[{\cite[Conjecture~1.1]{kashiwara2024braid}}]
\label{con_main}
For every positive braid group element \(b\in \Br^+\), the algebra
\(\widehat{\mathcal A}(b)\) admits a quantum cluster algebra structure
whose cluster monomials belong to the global basis.
\end{conjecture}
 In simply-laced Dynkin type, Conjecture~\ref{con_main} was proved by Qin
\cite{qin2024analog} using the theory of based quantum cluster algebras.
Independently, Kashiwara--Kim--Oh--Park~\cite{kashiwara2025monoidal} proved the conjecture in the same
setting by means of monoidal categorification and module categories of quantum affine algebras.

The purpose of the present paper is to give a new and direct proof of this
cluster structure result, based on Lusztig parametrizations of the global
basis. Our method is different from both Qin's approach via based quantum
cluster algebras and the categorical approach of Kashiwara--Kim--Oh--Park.
We work directly with
the global basis of the bosonic extension. The key point is to compare the
Lusztig parametrizations attached to different expressions of a positive
braid group element and to prove that the corresponding transition maps
are compatible with cluster mutations.\\

\noindent
We now describe our main results. Let
\[
\mathbf i=(i_1,\dots,i_\ell)
\]
be an expression of \(b\in\Br^+\). For \(1\leq a\leq c\leq \ell\), we say
that \([a,c]\) is an \(i\)-box if \(i_a=i_c\). Associated with such boxes
are generalized quantum minors
\[
D_{\mathbf i}[a,c].
\]
In particular, the initial cluster variables attached to the word
\(\mathbf i\) are written in the form
\[
D_{\mathbf i}[s,\ell\},
\qquad
1\leq s\leq \ell.
\]
Together with these minors, one has an exchange matrix \(B_{\mathbf i}\),
a compatible skew-symmetric matrix \(\Lambda_{\mathbf i}\), and the set of
exchangeable indices
\[
K^{\mathrm{ex}}
=
\{s\in[1,\ell]\mid s^-\neq -\infty\},
\qquad
s^-:=\max\{r<s\mid i_r=i_s\}\cup\{-\infty\}.
\]

Our first main theorem proves that the cluster structure is independent of
the expression of \(b\).

\begin{theorem}[Theorem~\ref{theo_words}]
\label{theo_intro_main}
Assume that
\[
c_{ij}c_{ji}\leq 1
\qquad (i\neq j).
\] Suppose that
\(\widehat{\mathcal A}(b)\) admits a quantum cluster algebra structure
with initial seed
\[
\mathbf t_{\mathbf i}
=
\Bigl(
(D_{\mathbf i}[s,\ell\})_{1\leq s\leq \ell},
\Lambda_{\mathbf i},
B_{\mathbf i},
K^{\mathrm{ex}}
\Bigr),
\]
and suppose that all cluster monomials belong to the global basis of
\(\widehat{\mathcal A}(b)\). Then the same statement holds for every
expression of \(b\).
\end{theorem}
 Under this assumption, all braid relations are generated by \(2\)-moves
and \(3\)-moves.

The proof of Theorem~\ref{theo_intro_main} is based on an explicit
description of Lusztig transition maps under \(2\)-moves and \(3\)-moves.
For a \(2\)-move, the transition map simply permutes the corresponding two
coordinates. For a \(3\)-move
\[
(i,j,i)\longleftrightarrow (j,i,j),
\qquad
c_{ij}c_{ji}=1,
\]
the transition map is the usual piecewise-linear rank-two transformation.
We prove that these transition maps send the Lusztig parameters of the
cluster variables for one expression to those for the other expression.
Consequently, braid moves correspond exactly to cluster mutations and
permutations of seeds.

Our second main result proves the conjecture in simply-laced Dynkin type
and, at the same time, establishes quantum \(T\)-system relations for the
generalized quantum minors.

\begin{theorem}[Theorem~\ref{theo_Dynkincase}; Theorem~\ref{thm_Tsystem}]
\label{theo_conmain}
Assume that \(C\) is of simply-laced Dynkin type. Then, for every positive
braid group element \(b\in\Br^+\), the algebra
\(\widehat{\mathcal A}(b)\) admits a quantum cluster algebra structure
whose cluster monomials belong to the global basis.

Moreover, let \(\mathbf i=(i_1,\dots,i_\ell)\) be an expression of
\(b\), and let \([a,c]\) be an \(i\)-box. Then the generalized quantum
minors satisfy the quantum \(T\)-system relation
\[
D_{\mathbf i}[a^+,c]\,
D_{\mathbf i}[a,c^-]
=
q^A
D_{\mathbf i}[a,c]\,
D_{\mathbf i}[a^+,c^-]
+
q^B\prod_{\substack{j\in I\\ c_{j i_a}=-1}}
D_{\mathbf i}[a(j)^+,c(j)^-],
\]
for some integers \(A,B\). Here we use the convention that
\[
D_{\mathbf i}[u,v]=1
\qquad\text{if }u>v.
\]
\end{theorem}

The quantum \(T\)-system is a central component of our result. It gives an
explicit exchange relation for generalized quantum minors under cluster
mutation. Consequently, these minors are not merely distinguished global
basis elements; they also occur as cluster variables. A key advantage of our approach is that the \(T\)-system is obtained
directly from Lusztig parameters. More precisely, it follows from the
comparison of Lusztig parametrizations together with a local analysis of
the exchange quiver. We also formulate a conjectural quantum \(T\)-system for generalized
quantum minors in the general symmetric Kac--Moody setting; see
Conjecture~\ref{con_Tsystem}.

In this paper, we develop a Lusztig-parametrization approach to the global
basis of bosonic extension algebras. Thus the main novelty of this paper lies in an explicit
Lusztig-parametrization description of the change of cluster seeds under
braid moves. This allows us to prove that the resulting cluster structure on
\(\widehat{\mathcal A}(b)\) is independent of the chosen expression of
the braid group element \(b\).
Furthermore, we construct explicit mutation sequences which realize the
generalized quantum minors as cluster variables. As a result, we obtain the quantum
$T$-system for these minors and prove
Conjecture~\ref{con_main} for simply-laced Dynkin type.

 Although the existence theorem is proved here in simply-laced Dynkin type,
the Lusztig-parametrization part of the argument is formulated in a way
that does not rely on finite type. This suggests a possible extension to
symmetric Kac--Moody type. 

Moreover, the explicit mutation sequences we constructed also give the quantum
\(T\)-system for the corresponding quantum minors in more general cases. Thus, in this framework,
Conjecture~\ref{con_main} implies Conjecture~\ref{con_Tsystem} in a more
general setting.

\subsection*{Organization}

The paper is organized as follows. In Section~\ref{sec_quantum}, we recall
basic material on quantum unipotent coordinate rings, dual PBW bases,
dual canonical bases, Lusztig parametrizations, and quantum cluster
algebras. In Section~\ref{sec_bosonic}, we review bosonic extensions of
quantum coordinate rings, their global bases, braid group symmetries, and
PBW bases. We then introduce the quantum minors associated with expressions
of positive braid group elements and formulate the conjectural cluster
structure on \(\widehat{\mathcal A}(b)\).

The next part of Section~\ref{sec_bosonic} is devoted to Lusztig transition
maps and their compatibility with braid moves. This proves that the
cluster structure is independent of the chosen expression. We then study
explicit mutation sequences on the exchange quivers and use them to show
that generalized quantum minors occur as cluster variables. This also
gives the quantum \(T\)-system for these minors. Finally, in
Section~\ref{sec_monoidal}, we recall the categorical realization in
simply-laced Dynkin type and combine it with the preceding results to
prove Theorem~\ref{theo_conmain}.

\subsection*{Acknowledgements}
The author is grateful to Ryo Fujita and Masaki Kashiwara for their valuable support, encouragement, and mathematical guidance during the preparation of this work.

\section{Preliminaries}\label{sec_quantum}

Let $C = (c_{ij})$ be a generalized Cartan matrix of size $I \times I$, and let $\mathfrak{g}$ be the associated Kac--Moody Lie algebra. We denote by $R^+$ the set of positive roots, by $Q$ (resp. $Q^+$) the root lattice (resp. positive root lattice), and by $\alpha_i$ (resp. $\alpha_i^\vee$) the simple roots (resp. simple coroots) for $C$. The fundamental weights are denoted by $\varpi_i$, and the weight lattice is defined as $P = \mathbb{Z}[\varpi_i]_{i \in I}$. We choose a diagonal matrix $D = \operatorname{diag}(d_i)_{i \in I}$ such that $DC = (a_{ij})$ is symmetric and $\min\{d_i \mid i \in I\} = 1$. A bilinear form $(\cdot, \cdot): P \times P \to \mathbb{Z}$ is defined by $(\alpha_i, \alpha_j) = a_{ij}$. Additionally, we set
\[
\langle h, \alpha_i \rangle = \frac{2(h, \alpha_i)}{(\alpha_i, \alpha_i)} \quad \text{for any } h \in P.
\]
The Weyl group $W$ of $\mathfrak{g}$ is generated by simple reflections $s_i$, which act on $P$ via
\[
s_i(\lambda) = \lambda - \langle \lambda, \alpha_i \rangle \alpha_i.
\]

Let $\mathbb{K}$ be either the field $\mathbb{Q}(q^{1/2})$ or the ring $\mathbb{Z}[q^{\pm 1/2}]$. For each $i \in I$, define $q_i = q^{d_i}$ and the quantum integers as follows:
\[
[n]_i = \frac{q_i^n - q_i^{-n}}{q_i - q_i^{-1}}, \quad [n]_i! = [n]_i [n-1]_i \cdots [1]_i, \quad \text{and} \quad \begin{bmatrix}
n\\ k
\end{bmatrix}_i = \frac{[n]_i!}{[k]_i! [n - k]_i!}.
\]

\subsection{Quantum unipotent coordinate rings}

Let $C=(c_{ij})_{i,j\in I}$ be a symmetrizable generalized Cartan matrix, and let $U_q(\mathfrak g)$ be the corresponding quantum group over $\mathbb Q(q^{1/2})$. Recall that $U_q(\mathfrak g)$ is generated by
\[
\{e_i,f_i\mid i\in I\}\cup \{q^h\mid h\in P\},
\]
subject to the standard Drinfeld--Jimbo relations:
\begin{align*}
&q^0=1,
\qquad
q^hq^{h'}=q^{h+h'},
\\
&q^he_iq^{-h}=q^{\langle h,\alpha_i\rangle}e_i,
\qquad
q^hf_iq^{-h}=q^{-\langle h,\alpha_i\rangle}f_i,
\\
&e_if_j-f_je_i=\delta_{ij}\frac{t_i-t_i^{-1}}{q_i-q_i^{-1}},
\qquad
t_i=q^{d_i\alpha_i^\vee},
\\
&\sum_{r=0}^{1-c_{ij}}(-1)^r
\begin{bmatrix}
1-c_{ij}\\ r
\end{bmatrix}_i
 e_i^{1-c_{ij}-r}e_je_i^r=0
\qquad (i\neq j),
\\
&\sum_{r=0}^{1-c_{ij}}(-1)^r
\begin{bmatrix}
1-c_{ij}\\ r
\end{bmatrix}_i
 f_i^{1-c_{ij}-r}f_jf_i^r=0
\qquad (i\neq j).
\end{align*}

Let $U_q(\mathfrak n)$ be the subalgebra generated by $\{e_i\mid i\in I\}$. It is naturally $Q^+$-graded:
\[
U_q(\mathfrak n)=\bigoplus_{\alpha\in Q^+}U_q(\mathfrak n)_\alpha.
\]
The quantum unipotent coordinate ring is defined as the graded dual
\[
A_q(\mathfrak n)
:=
\bigoplus_{\alpha\in Q^+}
\operatorname{Hom}_{\mathbb Q(q^{1/2})}(U_q(\mathfrak n)_\alpha,\mathbb Q(q^{1/2})).
\]
When $q=1$, the specialization $A(\mathfrak n)$ is canonically isomorphic to the coordinate ring $\mathbb C[N]$ of the maximal unipotent subgroup $N$ associated with $\mathfrak g$. Hence $A_q(\mathfrak n)$ is called the \emph{quantum unipotent coordinate ring}.

\subsubsection{Dual PBW basis}

Fix a reduced expression
\[
\mathbf{i}=(i_1,\dots,i_\ell)
\]
of a Weyl group element $w\in W$. Associated with $\mathbf{i}$ is the sequence of positive roots
\[
\beta_k:=s_{i_1}\cdots s_{i_{k-1}}(\alpha_{i_k}),
\qquad
1\le k\le \ell.
\]
The corresponding root vectors are denoted by $E^*(\beta_k)$. Let

\[
R(w)=\{\beta_k\mid 1\le k\le \ell\},
\]
and this set is independent of the chosen reduced expression.
Define
\[
\mathfrak n(w):=\bigoplus_{\alpha\in R(w)}\mathfrak n_\alpha.
\]
The algebra $A_q(\mathfrak n(w))$ is the subalgebra of $A_q(\mathfrak n)$ generated by
\[
\{E^*(\beta)\mid \beta\in R(w)\}.
\]
\begin{definition}\label{def_order}
For
\[
\mathbf a=(a_1,\dots,a_\ell),
\qquad
\mathbf b=(b_1,\dots,b_\ell),
\]
we define three partial orders as follows.
\begin{enumerate}
    \item We write $\mathbf a\prec_l \mathbf b$ if there exists $k$ such that
\[
a_j=b_j \quad \text{for all } j<k,
\qquad
a_k<b_k.
\]
Equivalently, $\mathbf a$ is smaller than $\mathbf b$ in the left
lexicographic order.
\item Similarly, we write $\mathbf a\prec_r \mathbf b$ if there exists $p$ such that
\[
a_i=b_i \quad \text{for all } i>p,
\qquad
a_p<b_p.
\]
Equivalently, $\mathbf a$ is smaller than $\mathbf b$ in the right
lexicographic order.
\item Finally, we write
\[
\mathbf a\prec \mathbf b
\]
if both $\mathbf a\prec_l\mathbf b$ and $\mathbf a\prec_r\mathbf b$ hold.
\end{enumerate}
\end{definition}

For $\mathbf a=(a_1,\dots,a_\ell)\in \mathbb Z_{\ge0}^\ell$, define
\[
E^*(\mathbf i, \mathbf a)
:=
\Bigl(\prod_{k=1}^\ell q^{a_k(a_k-1)}\Bigr)
E^*(\beta_\ell)^{a_\ell}\cdots E^*(\beta_1)^{a_1}.
\]

\begin{theorem}[{\cite{lusztig2010introduction}}]
The set
\[
\{E^*(\mathbf i, \mathbf a)\mid \mathbf a\in \mathbb Z_{\ge0}^\ell\}
\]
forms a basis of $A_q(\mathfrak n(w))$, called the \emph{dual PBW basis}.
\end{theorem}

\subsubsection{Dual canonical basis}

Let $\overline{\phantom{x}}$ be the bar involution on $A_q(\mathfrak n)$ defined by
\[
\overline q=q^{-1},
\qquad
\overline{\langle i\rangle}=\langle i\rangle.
\]

\begin{theorem}[{\cite{lusztig2010introduction}}]
\label{theo_canonical}
For every reduced expression $\mathbf i$ of $w$, there exists a unique basis
\[
\{B(\mathbf i,\mathbf a)\mid \mathbf a\in \mathbb Z_{\ge0}^\ell\}
\]
of $A_q(\mathfrak n(w))$ satisfying:
\begin{enumerate}
\item $\overline{B(\mathbf i,\mathbf a)}=B(\mathbf i,\mathbf a)$;
\item
\[
E^*(\mathbf i, \mathbf a)
=
B(\mathbf i,\mathbf a)
+
\sum_{\mathbf a'\prec \mathbf a}
f_{\mathbf a'}(q)B(\mathbf i,\mathbf a'),
\]
where $f_{\mathbf a'}(q)\in q\mathbb Z[q]$.
\end{enumerate}
The basis $\{B(\mathbf i,\mathbf a)\}$ is called the \emph{dual canonical basis}. We denote by $\BB(w)$ the set of dual canonical basis of $A_q(\fn(w))$.
\end{theorem}

\begin{lemma}
\label{lem_prodcanonical}
For any $\mathbf a,\mathbf b\in \mathbb Z_{\ge0}^\ell$,
\[
B(\mathbf i,\mathbf a)B(\mathbf i,\mathbf b)
=
q^AB(\mathbf i,\mathbf a+\mathbf b)
+
\sum_{\mathbf c\prec \mathbf a+\mathbf b}
 g_{\mathbf c}(q)B(\mathbf i, \mathbf c),
\]
where $g_{\mathbf c}(q)\in \mathbb Z[q^{\pm}]$.
\end{lemma}

\begin{proof}
By the triangular relation between the dual PBW basis and the dual
canonical basis, equivalently, since the transition matrix is unitriangular, we may invert
the triangular relation and obtain
\[
B(\mathbf i,\mathbf a)
=
E^*(\mathbf i,\mathbf a)
+
\sum_{\mathbf a'\prec \mathbf a}
u_{\mathbf a'}(q)E^*(\mathbf i,\mathbf a')
\]
for some \(u_{\mathbf a'}(q)\in \mathbb Z[q^{\pm 1}]\). Similarly,
\[
B(\mathbf i,\mathbf b)
=
E^*(\mathbf i,\mathbf b)
+
\sum_{\mathbf b'\prec \mathbf b}
v_{\mathbf b'}(q)E^*(\mathbf i,\mathbf b').
\]

Therefore the highest PBW term of
\(B(\mathbf i,\mathbf a)B(\mathbf i,\mathbf b)\) is the same as the
highest PBW term of
\(E^*(\mathbf i,\mathbf a)E^*(\mathbf i,\mathbf b)\).

By the Levendorskii--Soibelman formula, we have
\begin{equation}\label{eq_prod_pbw}
E^*(\mathbf i,\mathbf a)E^*(\mathbf i,\mathbf b)
=
q^A E^*(\mathbf i,\mathbf a+\mathbf b)
+
\sum_{\mathbf c\prec \mathbf a+\mathbf b}
h_{\mathbf c}(q)E^*(\mathbf i,\mathbf c).
\end{equation}
Thus the coefficient of the highest PBW term
\(E^*(\mathbf i,\mathbf a+\mathbf b)\) in
\(B(\mathbf i,\mathbf a)B(\mathbf i,\mathbf b)\) is \(q^A\).

On the other hand, write
\[
B(\mathbf i,\mathbf a)B(\mathbf i,\mathbf b)
=
\sum_{\mathbf c\preccurlyeq \mathbf a+\mathbf b}
f_{\mathbf c}(q)B(\mathbf i,\mathbf c).
\]
Since
\[
B(\mathbf i,\mathbf a+\mathbf b)
=
E^*(\mathbf i,\mathbf a+\mathbf b)
+
\sum_{\mathbf c\prec \mathbf a+\mathbf b}
w_{\mathbf c}(q)E^*(\mathbf i,\mathbf c),
\]
and every \(B(\mathbf i,\mathbf c)\) with
\(\mathbf c\prec \mathbf a+\mathbf b\) has highest PBW term strictly lower
than \(E^*(\mathbf i,\mathbf a+\mathbf b)\), the coefficient of
\(E^*(\mathbf i,\mathbf a+\mathbf b)\) in the PBW expansion of
\(B(\mathbf i,\mathbf a)B(\mathbf i,\mathbf b)\) is exactly
\(f_{\mathbf a+\mathbf b}(q)\). Hence
\[
f_{\mathbf a+\mathbf b}(q)=q^A.
\]
This proves
\[
B(\mathbf i,\mathbf a)B(\mathbf i,\mathbf b)
=
q^A B(\mathbf i,\mathbf a+\mathbf b)
+
\sum_{\mathbf c\prec \mathbf a+\mathbf b}
g_{\mathbf c}(q)B(\mathbf i,\mathbf c),
\]
as desired.
\end{proof}

\subsubsection{Transition maps of Lusztig parameters}

Let
\[
\Phi_{\mathbf i}:\mathbb Z_{\ge0}^\ell\longrightarrow \BB(w)
\]
be the parametrization defined by
\[
\Phi_{\mathbf i}(\mathbf a)=B(\mathbf i,\mathbf a).
\]
Its inverse is denoted by
\[
\mathcal L_{\mathbf i}:\BB(w)\longrightarrow \mathbb Z_{\ge0}^\ell,
\]
and is called the \emph{Lusztig parametrization} associated with $\mathbf i$.

Any two reduced expressions of $w$ are connected by braid moves:
\begin{enumerate}
\item $2$-move: $ij\leftrightarrow ji$ if $d(i,j)=0$;
\item $3$-move: $iji\leftrightarrow jij$ if $c_{ij}c_{ji}=1$;
\item $4$-move: $ijij\leftrightarrow jiji$ if $c_{ij}c_{ji}=2$;
\item $6$-move: $ijijij\leftrightarrow jijiji$ if $c_{ij}c_{ji}=3$.
\end{enumerate}
In this paper, we only consider $2$-, $3$-moves.

For two reduced expressions $\mathbf i$ and $\mathbf i'$ of $w$, define the transition map
\[
\Phi_{\mathbf i}^{\mathbf i'}
:=
\Phi_{\mathbf i'}^{-1}\circ \Phi_{\mathbf i}.
\]

\begin{theorem}[{\cite[Proposition~5.2]{kamnitzer2010mirkovic}}]
\label{theo_moves}
Let $\mathbf a=(a_1,\dots,a_\ell)\in \mathbb Z_{\ge0}^\ell$.
\begin{enumerate}
\item Suppose that $\mathbf i'$ is obtained from $\mathbf i$ by a $2$-move at positions $(k,k+1)$. Then
\[
\Phi_{\mathbf i}^{\mathbf i'}(\mathbf a)_i
=
\begin{cases}
a_{k+1} & i=k,
\\
a_k & i=k+1,
\\
a_i & \text{otherwise}.
\end{cases}
\]

\item Suppose that $\mathbf i'$ is obtained from $\mathbf i$ by a $3$-move at positions $(k-1,k,k+1)$. Then
\[
\Phi_{\mathbf i}^{\mathbf i'}(\mathbf a)_i
=
\begin{cases}
a_k+a_{k+1}-p & i=k-1,
\\
p & i=k,
\\
a_{k-1}+a_k-p & i=k+1,
\\
a_i & \text{otherwise},
\end{cases}
\]
where $p=\min(a_{k-1},a_{k+1})$.
\end{enumerate}
\end{theorem}

\subsection{Quantum cluster algebras}
In this section, we recall the definition of quantum cluster algebras.
\subsubsection{Definition of quantum cluster algebras}
Let \( K = [1, r] \) be a finite set with a partition \( K = K^{\mathrm{ex}} \cup K^{\mathrm{fr}} \). Let \( B \) be an integer-valued \( K \times K^{\mathrm{ex}} \)-matrix whose principal part, \( B_{K^{\mathrm{ex}} \times K^{\mathrm{ex}}} \), is skew-symmetrizable. That is, there exists a diagonal matrix \( D' = \operatorname{diag}(d'_i) \) such that:
\[
d'_i b_{ij} = -d'_j b_{ji} \quad \text{for all } i, j \in K^{\mathrm{ex}}.
\]

Additionally, let \( \Lambda \) be a skew-symmetric \( K \times K \)-matrix. The pair \( (\Lambda, B) \) is called \emph{compatible} if it satisfies:
\[
\sum_{k \in K}  b_{ki}\lambda_{kj} = 2 d'_i \delta_{ij} \quad \text{for all } i \in K^{\mathrm{ex}}, \, j \in K.
\]

Given a skew-symmetric matrix \( \Lambda = (\lambda_{ij}) \), we define the \emph{quantum torus} \( \mathcal{T}_\Lambda \) as the algebra \( \mathbb{K}[X_i^{\pm 1}]_{i \in K} \), where \( \mathbb{K} = \mathbb{Z}[q^{\pm 1/2}] \), with relations:
\[
X_i X_j = q^{\lambda_{ij}} X_j X_i, \quad X_i X_i^{-1} = X_i^{-1} X_i = 1.
\]

For any vector \( \mathbf{a} = (a_1, \ldots, a_r) \in \mathbb{Z}^r \), we define the monomial:
\[
X^{\mathbf{a}} = q^{\frac{1}{2} \sum_{i > j} a_i a_j \lambda_{ij}} X_1^{a_1} \cdots X_r^{a_r}.
\]

A \emph{seed} is a tuple \( \mathbf{t} := \{ (X_i)_{i \in K}, \Lambda, B, K^{\mathrm{ex}} \} \), where \( (\Lambda, B) \) is a compatible pair. The variables \( X_i \) are called \emph{cluster variables}.

\noindent

Given a compatible pair $(\Lambda, B)$ and an element
$k\in K^{\operatorname{ex}}$, we define a new pair
\[
\mu_k(\Lambda,B)
=
(\mu_k\Lambda,\mu_k B)
:=
(E^T\Lambda E,E B F),
\]
where the matrices
$E=(e_{ij})_{i,j\in K}$ and
$F=(f_{ij})_{i,j\in K^{\operatorname{ex}}}$ are given by
\begin{equation}\label{eq:mutation-matrices}
e_{ij}:=
\begin{cases}
\delta_{ij}, & \text{if } j\ne k,\\
-1, & \text{if } i=j=k,\\
\max(0,-b_{ik}), & \text{if } i\ne j=k,
\end{cases}
\qquad
f_{ij}:=
\begin{cases}
\delta_{ij}, & \text{if } i\ne k,\\
-1, & \text{if } i=j=k,\\
\max(0,b_{kj}), & \text{if } i=k\ne j,
\end{cases}
\end{equation}

For \( k \in K^{\mathrm{ex}} \), the \emph{mutation} of the cluster variable $X_i$ at \( k \) is given by
\[
\mu_k(X_i) = \begin{cases} 
    X_i & \text{if } i \neq k, \\
    X^{\mathbf{a}} + X^{\mathbf{a}'} & \text{if } i = k,
\end{cases}
\]

where:
\[
\mathbf{a} = ([b_{1k}]_+, \ldots, [b_{k-1,k}]_+, -1, [b_{k+1,k}]_+, \ldots, [b_{rk}]_+),
\]
\[
\mathbf{a}' = ([-b_{1k}]_+, \ldots, [-b_{k-1,k}]_+, -1, [-b_{k+1,k}]_+, \ldots, [-b_{rk}]_+),
\]

and \( [a]_+ = \max\{0, a\} \).

It can be verified that \( (\mu_k(\Lambda), \mu_k(B)) \) remains a compatible pair, yielding a new seed:
\[
\mu_k(\mathbf{t}) := \{ (\mu_k(X_i))_{i \in K}, \mu_k(\Lambda), \mu_k(B), K^{\mathrm{ex}} \}.
\]

Let \( T\) denote the set of all seeds obtained from \( \mathbf{t} \) by any finite sequence of mutations.

\begin{definition}
For a seed \( \mathbf{t} \), the \emph{quantum cluster algebra} \( \mathcal A(\mathbf{t}) \) is the \( \mathbb{K} \)-subalgebra of \( \mathcal{T}_\Lambda \) generated by all cluster variables \( X_i(\mathbf{t}') \) for all seeds \( \mathbf{t}' \in T \). Note that, in our convention, frozen variables are not assumed to be invertible.
\end{definition}

\subsubsection{Morphisms of quantum cluster algebras}

Let \( \sigma: K \to K \) be a permutation of the set \( K \) such that \( \sigma(K^{\mathrm{ex}}) \subset K^{\mathrm{ex}} \). We define the matrices \( B_\sigma \) and \( \Lambda_\sigma \) as follows:
\[
b_{\sigma, ij} = b_{\sigma(i), \sigma(j)}, \quad \lambda_{\sigma, ij} = \lambda_{\sigma(i), \sigma(j)}.
\]

The tuple \( \sigma(\mathbf{t}) := \{ (X_{\sigma(i)})_{i \in K}, \Lambda_\sigma, B_\sigma, K^{\mathrm{ex}} \} \) forms a seed, and there exists an isomorphism:
\begin{equation}\label{eq_rhocluster}
\mathcal A(\sigma(\mathbf{t})) \cong \mathcal A(\mathbf{t}).
\end{equation}

For a seed \( \mathbf{t} = \{ (X_i)_{i \in K}, B_{K \times K^{\mathrm{ex}}}, \Lambda_{K \times K}, K^{\mathrm{ex}} \} \), let \( J \subset K \) and \( J^{\mathrm{ex}} \subset K^{\mathrm{ex}} \). Suppose the submatrix satisfies:
\begin{equation}\label{eq_BKJ}
B_{(K \setminus J) \times J^{\mathrm{ex}}} = 0.
\end{equation}
It follows that $(B_{J\times J^{\operatorname{ex}}},\Lambda_{J\times J})$ is a compatible pair.
We define the \emph{restricted seed}:
\[
\mathbf{t}_J := \{ (X_i)_{i \in J}, B_{J \times J^{\mathrm{ex}}}, \Lambda_{J \times J}, J^{\mathrm{ex}} \}.
\]
A direct verification gives the following proposition.
\begin{proposition}\label{pro:subcluster}
There exists a quantum cluster algebra embedding:
\[
\mathcal A(\mathbf{t}_J) \subset \mathcal A(\mathbf{t}).
\]
\end{proposition}

\section{Bosonic extensions of quantum coordinate rings}
\label{sec_bosonic}

In this section, we recall the construction of bosonic extensions of quantum coordinate rings following \cite{kashiwara2024global,kashiwara2024braid}. We first review the definition of the bosonic extension and its global basis, and then recall the braid group action and the associated PBW basis.

\subsection{Bosonic extensions}

Let $C=(c_{ij})_{i,j\in I}$ be a symmetrizable generalized Cartan matrix, and let $\widehat{\mathcal A}$ denote the bosonic extension associated with $C$.

\begin{definition}[{\cite{kashiwara2024global}}]
The algebra $\widehat{\mathcal A}$ is the $\mathbb K$-algebra generated by
\[
\{x_{i,n}\mid i\in I,\ n\in \mathbb Z\},
\]
subject to the relations
\begin{align}
\sum_{r+s=-c_{ij}}(-1)^r&
\begin{bmatrix}
1-c_{ij}\\ r
\end{bmatrix}_{i}
 x_{i,n}^r x_{j,n} x_{i,n}^s =0,
\label{eq_bosonic_serre}
\\
 x_{i,n}x_{j,n+1}
 &=
 q^{(\alpha_i,\alpha_j)}x_{j,n+1}x_{i,n}
 +\delta_{ij}(1-q_i^2),
\label{eq_bosonic_adjacent}
\\
 x_{i,n}x_{j,m}
 &=
 q^{(-1)^{m-n+1}(\alpha_i,\alpha_j)}x_{j,m}x_{i,n}
 \qquad (m-n>1).
\label{eq_bosonic_far}
\end{align}
\end{definition}

The algebra $\widehat{\mathcal A}$ is naturally $Q^+$-graded by
\[
\wt(x_{i,n})=(-1)^n\alpha_i.
\]

For an interval $[a,c]\subset \mathbb Z$, let $\widehat{\mathcal A}[a,c]$ denote the subalgebra generated by
\[
\{x_{i,n}\mid i\in I,\ a\le n\le c\}.
\]
We write $\widehat{\mathcal A}[k]:=\widehat{\mathcal A}[k,k]$.

For elements $(x_c,\dots,x_a)$, define the ordered product
\[
\overrightarrow{\prod}_{k\in[a,c]}x_k
:=
 x_cx_{c-1}\cdots x_a.
\]

\begin{proposition}[{\cite[Corollary~4.4]{kashiwara2024global}}]
\label{prop_tensor_decomposition}
The following statements hold.
\begin{enumerate}
\item For every $k\in \mathbb Z$, there is an algebra isomorphism
\[
\widehat{\mathcal A}[k]\simeq A_q(\mathfrak n).
\]

\item For every interval $[a,c]$, the multiplication map induces a vector space isomorphism
\[
\widehat{\mathcal A}[c]\otimes \widehat{\mathcal A}[c-1]\otimes \cdots \otimes \widehat{\mathcal A}[a]
\xrightarrow{\sim}
\widehat{\mathcal A}[a,c],
\]
\[
 x_c\otimes x_{c-1}\otimes \cdots \otimes x_a
 \longmapsto
 \overrightarrow{\prod}_{k\in[a,c]}x_k.
\]
\end{enumerate}
\end{proposition}

We next recall the bar involution and the twisted duality map on $\widehat{\mathcal A}$. Define a $\mathbb Q$-algebra anti-automorphism
\[
\overline{\phantom{x}}: \widehat{\mathcal A}\to \widehat{\mathcal A}
\]
by
\[
\overline{q^{\pm1/2}}=q^{\mp1/2},
\qquad
\overline{x_{i,n}}=x_{i,n}.
\]

For a homogeneous element $x\in \widehat{\mathcal A}$, define
\begin{equation}
\label{eq_twisted_bar}
 c(x)
 :=
 q^{(\wt(x),\wt(x))/2}\overline{x}.
\end{equation}

We also define the shift automorphism
\begin{equation}
\label{eq_shift_automorphism}
\mathscr D(x_{i,n})=x_{i,n+1}.
\end{equation}

\subsubsection{Extended crystal basis}

Let $\mathbf B(\infty)$ be the dual canonical basis of $A_q(\mathfrak n)$. Define
\[
\widehat{\mathbf B}(\infty)
:=
\Bigl\{(b_k)_{k\in \mathbb Z}
\ \Big|\ 
 b_k\in \mathbf B(\infty),
\ b_k=1 \text{ for all but finitely many } k
\Bigr\}.
\]

The Kashiwara operators
\[
\widetilde e_i,\widetilde f_i,
\qquad
\widetilde e_i^*,\widetilde f_i^*
\]
act on $\widehat{\mathbf B}(\infty)$ and endow it with a crystal structure; see \cite{kashiwara2024global}.

For
\[
\mathbf b=(b_k)_{k\in \mathbb Z}\in \widehat{\mathbf B}(\infty),
\]
define
\[
P(\mathbf b)
:=
\overrightarrow{\prod}_{k\in \mathbb Z}
\psi_k\bigl(G^{\mathrm{up}}(b_k)\bigr),
\]
where
\[
\psi_k:A_q(\mathfrak n)\xrightarrow{\sim}\widehat{\mathcal A}[k]
\]
is the canonical isomorphism.

\begin{theorem}[{\cite[Theorem~6.6]{kashiwara2024global}}]
\label{thm_global_basis_hatA}
For every $\mathbf b\in \widehat{\mathbf B}(\infty)$, there exists a unique element
\[
G(\mathbf b)\in \widehat{\mathcal A}
\]
such that
\begin{align}
 c(G(\mathbf b))&=G(\mathbf b),
\\
 P(\mathbf b)-G(\mathbf b)
 &\in
 \sum_{\mathbf b'\prec \mathbf b}
 q\mathbb Z[q]P(\mathbf b').
\end{align}
Here $\prec$ denotes the order on $\widehat{\mathbf B}(\infty)$ introduced in \cite[Section~6]{kashiwara2024global}.
\end{theorem}

\subsubsection{Braid group symmetries}

We now recall the braid group symmetries on $\widehat{\mathcal A}$ introduced in \cite{kashiwara2024braid}.

For each $i\in I$, define an algebra automorphism
\[
T_i:\widehat{\mathcal A}\to \widehat{\mathcal A}
\]
by
\[
T_i(x_{j,n})=
\begin{cases}
 x_{i,n+1}
 & \text{if } j=i,
\\
\dfrac{
\sum_{r+s=-c_{ij}}(-1)^r q_i^{c_{ij}/2+r}
 x_{i,n}^{(s)}x_{j,n}x_{i,n}^{(r)}
}{(q_i^{-1}-q_i)^{-c_{ij}}}
 & \text{if } j\neq i.
\end{cases}
\]

Similarly, the inverse automorphism $T_i^{-1}$ is given by
\[
T_i^{-1}(x_{j,n})=
\begin{cases}
 x_{i,n-1}
 & \text{if } j=i,
\\
\dfrac{
\sum_{r+s=-c_{ij}}(-1)^r q_i^{c_{ij}/2+r}
 x_{i,n}^{(r)}x_{j,n}x_{i,n}^{(s)}
}{(q_i^{-1}-q_i)^{-c_{ij}}}
 & \text{if } j\neq i.
\end{cases}
\]

These operators satisfy the braid relations:
\[
T_iT_i^{-1}=T_i^{-1}T_i=\mathrm{Id},
\]
\[
\underbrace{T_i^{\pm}T_j^{\pm}\cdots}_{m_{ij}}
=
\underbrace{T_j^{\pm}T_i^{\pm}\cdots}_{m_{ij}}
\qquad (i\neq j),
\]
where
\[
m_{ij}=
\begin{cases}
 c_{ij}c_{ji}+2 & \text{if } c_{ij}c_{ji}\le2,
\\
6 & \text{if } c_{ij}c_{ji}=3,
\\
\infty & \text{otherwise}.
\end{cases}
\]

By construction, there exists a group morphism $T$ from positive braid group $\Br^+$ to the automorphism group $\operatorname{Aut}(\widehat{\cA})$. For braid group element $b\in \Br^+$, we denote by $T_b$ the image of $b$ under the morphism $T$.

\begin{theorem}[{\cite[Theorem~3.7]{kashiwara2024braid}}]
\label{thm_braid_on_basis}
For every $i\in I$, the automorphism $T_i$ induces a bijection
\[
T_i:\widehat{\mathbf B}(\infty)
\xrightarrow{\sim}
\widehat{\mathbf B}(\infty).
\]
\end{theorem}

\subsubsection{PBW basis}

Let $b\in \Br^+$ with the length $\ell=\ell(b)$ and fix an expression of $b$
\[
\mathbf i=(i_1,\dots,i_\ell).
\]
For $1\le k\le \ell$, define
\[
T_{b_{\le k}}:=T_{i_1}\cdots T_{i_k},
\qquad
T_{b_{\le0}}:=\mathrm{Id}.
\]

The associated root vectors are defined by
\[
E_k^{\bfi}
:=
 T_{b_{\le k-1}}(q_i^{1/2}x_{i_k,0}).
\]

\begin{definition}
The algebra $\widehat{\mathcal A}(b)$ is the subalgebra of $\widehat{\mathcal A}$ generated by
\[
\{E_k^{\bfi}\mid 1\le k\le \ell\}.
\]
\end{definition}

The algebra $\widehat{\mathcal A}(b)$ is independent of the choice of expression of $b$.

For
\[
\mathbf a=(a_1,\dots,a_\ell)\in \mathbb Z_{\ge0}^\ell,
\]
define the PBW monomial
\[
E(\bfi,\mathbf a)
:=
\Bigl(\prod_{k=1}^\ell q^{a_k(a_k-1)}\Bigr)
(E_\ell^{\bfi})^{a_\ell}\cdots (E_1^{\bfi})^{a_1}.
\]

\begin{proposition}[{\cite[Proposition~4.7 and Lemma~4.16]{kashiwara2024braid}}]
\label{prop_pbw_global}
Let $b\in \Br^+$. Then:
\begin{enumerate}
\item
\[
\widehat{\mathcal A}(b)
=
\widehat{\mathcal A}_{\ge0}
\cap
T_b(\widehat{\mathcal A}_{<0}).
\]

\item Given an expression $\bfi$ of $b$. The set
\[
\{E(\bfi,\mathbf a)\mid \mathbf a\in \mathbb Z_{\ge0}^\ell\}
\]
forms a basis of $\widehat{\mathcal A}(b)$, called the \emph{PBW basis}.

\item There exists a unique basis
\[
\{B(\mathbf i,\mathbf a)\mid \mathbf a\in \mathbb Z_{\ge0}^\ell\}
\]
of $\widehat{\mathcal A}(b)$ such that
\begin{align}\label{eq_pbw_global}
 c(B(\mathbf i, \mathbf a))&=B(\mathbf i,\mathbf a),
\\
 E(\bfi, \mathbf a)
 &=
 B(\bfi,\mathbf a)
 +
 \sum_{\mathbf c\prec \mathbf a}
 g_{\mathbf c}(q)B(\bfi,\mathbf c),
\end{align}
where $g_{\mathbf c}(q)\in q\mathbb Z[q]$.

This basis is called the \emph{global basis} of $\widehat{\mathcal A}(b)$, and it is contained in the extended crystal basis $\widehat{\mathbf B}(\infty)$. We denote by $\cL_\bfi(B)$ the $\bfi$-Lusztig parameter of the global basis element $B\in \widehat{\cA}(b)$. We denote by $\BB(b)$ the global basis of $\widehat{\cA}(b)$.
\end{enumerate}
\end{proposition}

\begin{lemma}\label{lem_prodglobal}
    For an expression $\bfi$ of $b\in \Br^+$, we have 
    \[B(\bfi,\mathbf{a})B(\bfi,\mathbf{b})=q^AB(\bfi,\bfa+\bb)+\sum_{\mathbf{c}\prec \bfa+\bb}f_{\mathbf{c}}(q)B(\bfi,\mathbf{c}).\]
\end{lemma}
\begin{proof}
  The proof is similar to that of Lemma~\ref{lem_prodcanonical}. Note that the PBW basis satisfies the Levendorskii--Soibelman formula
\eqref{eq_prod_pbw}.
\end{proof}

\begin{lemma}
\label{lem:Tu-global-basis}
Let $\bfi=(i_1,\dots,i_m)$ be an expression of
$b\in\operatorname{Br}^+$, and let $\bfj=(j_1,\dots,j_n)$ be an
expression of $u\in\operatorname{Br}^+$. Set
\[
\bfk:=\bfj\bfi=(j_1,\dots,j_n,i_1,\dots,i_m).
\]
Let $\bfa\in\mathbb Z_{\ge0}^{m}$ be regarded as a $\bfk$-parameter
supported on the subword $\bfi$. Then
\[
T_u\bigl(B(\bfi,\bfa)\bigr)=B(\bfk,\bfa).
\]
\end{lemma}

\begin{proof}
We first prove the statement for PBW monomials. For $1\le r\le m$, let
$\bfe_r$ be the vector whose $r$-th coordinate is $1$ and whose other
coordinates are $0$. Then
\[
B(\bfi,\bfe_r)=E_r^{\bfi}.
\]
By the definition of root vectors, applying the braid symmetry $T_u$
sends the root vector attached to the $r$-th position of $\bfi$ to the
corresponding root vector attached to the same position inside the
concatenated word $\bfk=\bfj\bfi$. Hence
\[
T_u(E_r^{\bfi})=E_r^{\bfk}.
\]
It follows multiplicatively that, for every parameter $\bfa$ supported on
the subword $\bfi$,
\begin{equation}
\label{eq:Tu-PBW}
T_u\bigl(E(\bfi,\bfa)\bigr)=E(\bfk,\bfa).
\end{equation}

We now pass from PBW monomials to global basis elements. By the triangular relation between PBW monomials
and global basis elements in Proposition~\ref{prop_pbw_global}(3), we have
\[
E(\bfi,\bfa)
=
B(\bfi,\bfa)
+
\sum_{\bfa'\prec\bfa}
g_{\bfa'}(q)B(\bfi,\bfa'),
\qquad
g_{\bfa'}(q)\in q\mathbb Z[q].
\]
Applying $T_u$ and using \eqref{eq:Tu-PBW}, we obtain
\[
E(\bfk,\bfa)
=
T_u(B(\bfi,\bfa))
+
\sum_{\bfa'\prec\bfa}
g_{\bfa'}(q)T_u(B(\bfi,\bfa')).
\]
Since each $T_u(B(\bfi,\bfa'))$ is again a global basis element by Theorem \ref{thm_braid_on_basis} and 
the triangular expansion of $E(\bfk,\bfa)$ with respect to the global
basis is unique, its leading term must be $B(\bfk,\bfa)$. Therefore
\[
T_u(B(\bfi,\bfa))=B(\bfk,\bfa),
\]
as desired.
\end{proof}

\subsubsection{Lusztig parameters of global basis}
For an expression \( \mathbf{i} \) of \( b \), we define the isomorphism:
\[
\Psi_{\mathbf{i}}: \mathbb{Z}_{\geq 0}^{[1, \ell(b)]} \to \BB(b),
\]
where \( \BB(b) \) denotes the global basis of \( \widehat{\mathcal A}(b) \). For another expression \( \mathbf{i}' \) of \( b \), we define the transition map:
\[
\Psi_{\mathbf{i}}^{\mathbf{i}'} := \Psi_{\mathbf{i}'}^{-1} \circ \Psi_{\mathbf{i}}: \mathbb{Z}_{\geq 0}^{[1, \ell(b)]} \to \mathbb{Z}_{\geq 0}^{[1, \ell(b)]}.
\]

\begin{theorem}\label{theo_braidmoves}
Let \(\mathbf i=(i_1,\dots,i_\ell)\) be an expression of
\(b\in\Br^+\), and let \(\mathbf i'\) be obtained from \(\mathbf i\) by a
single braid move. Then the transition map
\[
\Psi_{\mathbf i}^{\mathbf i'}:\mathbb Z_{\geq 0}^{\ell}
\longrightarrow
\mathbb Z_{\geq 0}^{\ell}
\]
is given as follows.

\begin{enumerate}
\item
If \(\mathbf i'\) is obtained from \(\mathbf i\) by a \(2\)-move at
positions \((k,k+1)\), with
\[
c_{i_k i_{k+1}}=0,
\]
then
\[
\Psi_{\mathbf i}^{\mathbf i'}(\mathbf a)_t=
\begin{cases}
a_{k+1} & t=k,\\
a_k & t=k+1,\\
a_t & \text{otherwise}.
\end{cases}
\]

\item
If \(\mathbf i'\) is obtained from \(\mathbf i\) by a \(3\)-move
\[
(i,j,i)\longleftrightarrow (j,i,j)
\]
at positions \((k-1,k,k+1)\), with
\[
c_{ij}c_{ji}=1,
\]
then
\[
\Psi_{\mathbf i}^{\mathbf i'}(\mathbf a)_t=
\begin{cases}
a_k+a_{k+1}-p & t=k-1,\\
p & t=k,\\
a_{k-1}+a_k-p & t=k+1,\\
a_t & \text{otherwise},
\end{cases}
\]
where
\[
p=\min(a_{k-1},a_{k+1}).
\]
\end{enumerate}
\end{theorem}

\begin{proof}
We consider the two possible braid moves.

\smallskip

\noindent
\textbf{The \(2\)-move case.}
Assume that \(\mathbf i'\) is obtained from \(\mathbf i\) by interchanging
\(i_k\) and \(i_{k+1}\), where
\[
c_{i_k i_{k+1}}=0.
\]
Then the corresponding root vectors satisfy
\[
E^{\mathbf i}_k=E^{\mathbf i'}_{k+1},
\qquad
E^{\mathbf i}_{k+1}=E^{\mathbf i'}_k.
\]
Moreover, these two root vectors \(q\)-commute. Hence, with the normalized
PBW monomial convention, one has
\[
E(\mathbf i,\mathbf a)
=
E\bigl(\mathbf i',\Psi_{\mathbf i}^{\mathbf i'}(\mathbf a)\bigr),
\]
where \(\Psi_{\mathbf i}^{\mathbf i'}\) simply interchanges the \(k\)-th
and \((k+1)\)-st coordinates. By the triangularity and uniqueness of the
global basis element with a given leading PBW term, we obtain
\[
B(\mathbf i,\mathbf a)
=
B\bigl(\mathbf i',\Psi_{\mathbf i}^{\mathbf i'}(\mathbf a)\bigr).
\]
This proves the formula in the \(2\)-move case.

\smallskip

\noindent
\textbf{The \(3\)-move case.}
Assume that \(\mathbf i'\) is obtained from \(\mathbf i\) by replacing
\[
(i,j,i)
\]
with
\[
(j,i,j)
\]
at positions \((k-1,k,k+1)\), where
\[
c_{ij}c_{ji}=1.
\]
Let
\[
\mathbf j=(i,j,i),
\qquad
\mathbf j'=(j,i,j).
\]
For the rank-two subword, Theorem~\ref{theo_moves} gives
\[
B(\mathbf j,(a_{k-1},a_k,a_{k+1}))
=
B\bigl(
\mathbf j',
(a_k+a_{k+1}-p,\ p,\ a_{k-1}+a_k-p)
\bigr),
\]
where
\[
p=\min(a_{k-1},a_{k+1}).
\]

We first assume that \(\mathbf a\) is supported on the block
\([k-1,k+1]\). Then
\[
E(\mathbf i,\mathbf a)
=
T_{b_{<k-1}}
\bigl(E(\mathbf j,(a_{k-1},a_k,a_{k+1}))\bigr),
\]
and similarly
\[
E\bigl(\mathbf i',\Psi_{\mathbf i}^{\mathbf i'}(\mathbf a)\bigr)
=
T_{b_{<k-1}}
\bigl(
E(\mathbf j',
(a_k+a_{k+1}-p,\ p,\ a_{k-1}+a_k-p))
\bigr).
\]
By Proposition~\ref{prop_pbw_global}(3) and
Theorem~\ref{thm_braid_on_basis}, the braid operator \(T_{b_{<k-1}}\)
sends the corresponding rank-two global basis element to the global basis
element in the full word. Therefore the rank-two identity implies
\[
B(\mathbf i,\mathbf a)
=
B\bigl(\mathbf i',\Psi_{\mathbf i}^{\mathbf i'}(\mathbf a)\bigr)
\]
for all \(\mathbf a\) supported on the block.

For a general \(\mathbf a\), write
\[
\mathbf a
=
\mathbf a_{<k-1}
+
\mathbf a_{[k-1,k+1]}
+
\mathbf a_{>k+1}.
\]
The root vectors outside the local block are unchanged by the \(3\)-move,
because the two subwords \((i,j,i)\) and \((j,i,j)\) represent the same
braid group element. Hence the corresponding global basis elements
supported outside the block are the same for \(\mathbf i\) and
\(\mathbf i'\).

By the triangular factorization property of the global basis with respect
to the PBW order, we have
\[
\begin{aligned}
&
B(\mathbf i,\mathbf a_{>k+1})
B(\mathbf i,\mathbf a_{[k-1,k+1]})
B(\mathbf i,\mathbf a_{<k-1})
\\
&\qquad =
B(\mathbf i,\mathbf a)
+
\sum_{\mathbf a'\prec \mathbf a}
f_{\mathbf a'}(q)B(\mathbf i,\mathbf a'),
\end{aligned}
\]
with \(f_{\mathbf a'}(q)\in \mathbb Z[q^\pm]\). Likewise,
\[
\begin{aligned}
&
B(\mathbf i',\mathbf a_{>k+1})
B\bigl(
\mathbf i',
\Psi_{\mathbf i}^{\mathbf i'}(\mathbf a_{[k-1,k+1]})
\bigr)
B(\mathbf i',\mathbf a_{<k-1})
\\
&\qquad =
B\bigl(\mathbf i',\Psi_{\mathbf i}^{\mathbf i'}(\mathbf a)\bigr)
+
\sum_{\mathbf b\prec \Psi_{\mathbf i}^{\mathbf i'}(\mathbf a)}
g_{\mathbf b}(q)B(\mathbf i',\mathbf b),
\end{aligned}
\]
with \(g_{\mathbf b}(q)\in \mathbb Z[q^\pm]\).

The outer factors in the two products agree, and the middle factors agree
by the rank-two case. Hence the two products are equal. By the uniqueness
of the leading term in the triangular global-basis expansion, their leading
global basis elements must agree. Therefore
\[
B(\mathbf i,\mathbf a)
=
B\bigl(\mathbf i',\Psi_{\mathbf i}^{\mathbf i'}(\mathbf a)\bigr).
\]
This proves the \(3\)-move formula.
\end{proof}

\subsection{Quantum cluster algebra associated with words}

In this section, we introduce quantum minors associated with expressions of braid group elements. These elements are expected to provide the cluster variables in the quantum cluster structures on Bosonic extension algebras.

Fix a word
\[
\mathbf i=(i_1,\dots,i_\ell)
\]
of $I$. We denote by $[1,\ell]$ the set of integers $k$ with $1\leq k\leq \ell$.  For $a\in[1,\ell]$, we call $i_a$ the \emph{color} of $a$. For
$1\le a\le c\le \ell$, the interval $[a,c]$ is called an \emph{$i$-box}
if $i_a=i_b$.

We use the following notation. For $a\in[1,\ell]$ and $j\in I$, set
\begin{align}
\label{eq:notationa+a-}
a^-&:=\max\bigl(\{k<a\mid i_k=i_a\}\cup\{-\infty\}\bigr),
&
a^+&:=\min\bigl(\{k>a\mid i_k=i_a\}\cup\{\ell+1\}\bigr),
\\
a(j)^-&:=\max\bigl(\{k<a\mid i_k=j\}\cup\{-\infty\}\bigr),
&
a(j)^+&:=\min\bigl(\{k>a\mid i_k=j\}\cup\{\ell+1\}\bigr),
\\
a_{\max}&:=\max\{k\in[1,\ell]\mid i_k=i_a\},
&
a_{\min}&:=\min\{k\in[1,\ell]\mid i_k=i_a\}.
\end{align}
Thus $a^-$ and $a^+$ are the predecessor and successor of $a$ among
vertices of the same color, while $a(j)^-$ and $a(j)^+$ are the
predecessor and successor of $a$ among vertices of color $j$.

For simplicity, we sometimes write
\[
a^+(j)^-:=(a^+)(j)^-.
\]
In other words, $a^+(j)^-$ denotes the predecessor of $a^+$ among
vertices of color $j$, whenever this vertex exists.

It is often convenient to write
\[
a=(i_a,n)
\]
if $a$ is the $n$-th occurrence of the color $i_a$ in the word $\bfi$.
With this convention, if $a=(i_a,n)$, then
\[
a^-=(i_a,n-1).
\]
More generally, we write
\[
a^{-k}:=(i_a,n-k)
\]
whenever this vertex exists.
Finally, for $j\in I$, let $n_j$ denote the number of vertices of color
$j$ in the word $\bfi$.

Let $\bfi=(i_1,\dots,i_\ell)$ be an expression of
$b\in \Br^+$. An interval $[a,c]\subset [1,\ell]$ is called an \emph{$i$-box} if
$i_a=i_c$. If $i_a\neq i_c$, we define $[a,c\}=[a,d]$ where $d$ is the maximal index less than $c$ with $i_d=i_a$. Associated with an $i$-box $[a,c]$, define
\[
\bfi[a,c]=(a_1,\dots,a_\ell)\in \mathbb Z_{\ge0}^\ell
\]
by
\[
a_k=
\begin{cases}
1 & \text{if } a\le k\le c \text{ and } i_k=i_a,
\\
0 & \text{otherwise}.
\end{cases}
\]

\begin{definition}
The \emph{quantum minor} associated with $[a,c]$ is the global basis element for vector $\bfi[a,c]$,
\[
 D_\bfi[a,c]:=B(\bfi,\bfi[a,c]).
\]
\end{definition}
For simplicity, we denote by $D_{\bfi,s}$ the quantum minor $D_\bfi[s,\ell(b)\}$.
Set
\[
K^{\mathrm{ex}}
=
\{k\in [1,\ell]\mid k^- \text{ exists}\},
\qquad
K^{\mathrm{fr}}
=
[1,\ell]\setminus K^{\mathrm{ex}}.
\]

We next recall the exchange matrix associated with $\bfi$. Define
\[
B_{\bfi}=(b_{kl})_{K\times K^{\operatorname{ex}}}
\]
by
\begin{equation}
\label{eq_braidmatrix}
b_{kl}=
\begin{cases}
1 & \text{if } l=k^-,
\\
-1 & \text{if } l=k^+,
\\
c_{i_k i_l}
& \text{if } l^-<k^-<l<k,
\\
-c_{i_k i_l}
& \text{if } k^-<l^-<k<l,
\\
0 & \text{otherwise}.
\end{cases}
\end{equation}
and 
\[\Lambda_\bfi=(\lambda_{ij})_{K\times K}\]
by 
\[\lambda_{kl}=-\lambda_{lk}=-(\varpi_{i_k}-w_{k}\varpi_{i_k},\varpi_{i_l}+w_l\varpi_{i_l}) \text{ for }k\leq l\]
By \cite[Proposition~1.2]{fujita2023isomorphisms}, we have $(B_\bfi,\Lambda_\bfi)$ is a compatible pair. The quiver $Q_{\bfi}$ is defined by the exchange matrix $B_{\bfi}=(b_{kl})$. Its vertex set is $[1,\ell]$, and for each pair $k,l\in[1,\ell]$ with $b_{kl}>0$, we draw $b_{kl}$ arrows from $k$ to $l$. If $c_{ij}c_{ji}\leq 1$, then the quiver $Q_{\bfi}$ has two types of arrows: \emph{horizontal arrows} $k\to k^-$, and\emph{ ordinary arrows} $k\to l$ satisfying $k^-<l^-<k<l$ and $c_{i_k i_l}=-1$.
\begin{remark}
\label{rem_BGLS}
More generally, for any infinite word $\bfi$ over $I$, one can define the matrix $B_{\bfi}$ by the same formula \eqref{eq_braidmatrix}. In this case all vertices are exchangeable.

 Our convention the compatible pair $(B_\bfi,\Lambda_\bfi)$ is the compatible pair $(-B_\bfi,-\Lambda_\bfi)$ in the sense of \cite{fujita2023isomorphisms}.
\end{remark}

The quantum minors are expected to satisfy quantum $T$-system relations.

\begin{conjecture}
\label{con_Tsystem}
Let $\bfi$ be an expression of $b\in \Br^+$. For every $i$-box $[a,c]$, one has
\begin{equation}\label{eq_Tsystem}
D_{\bfi}[a^+,c]D_{\bfi}[a,c^-]
=
q^A
D_{\bfi}[a,c]D_{\bfi}[a^+,c^-]
+
q^B
\prod_{d(i_a,j)=1}
D_{\bfi}[a(j)^+,c(j)^-],
\end{equation}
where $A,B\in \mathbb Z$.
\end{conjecture}

\begin{remark}\label{rem_Tsystem}
In simply-laced Dynkin type, Conjecture~\ref{con_Tsystem} is known in simply-laced Dynkin type by
\cite[Theorem~5.16]{kashiwara2025monoidal} and \cite{qin2024analog}. 
\end{remark}

\subsection{Cluster structures in Bosonic extensions}
We now present a conjecture concerning the cluster algebra structure of the algebra \( \widehat{\mathcal A}(b) \), generated by elements associated with the braid group element \( b \).

\begin{conjecture}\label{con_boscluster}
Let \( C \) be a Cartan matrix satisfying \( c_{ij} c_{ji} \leq 1 \) for all \( i \neq j \), and let \( b \in \Br^+ \) with an expression \( \mathbf{i} := (i_1, \ldots, i_{\ell(b)}) \). Then, \( \widehat{\mathcal A}(b) \) is a quantum cluster algebra with the initial seed:
\[
\mathbf{t}_\mathbf{i} := \{ (D_{\mathbf{i},s})_{s \in [1, \ell(b)]}, \Lambda_\bfi, B_{\bfi}, K^{\mathrm{ex}} \}.
\]
 Furthermore, the cluster monomials are contained in the global basis \( \mathbf{B}(b) \).
\end{conjecture}

\subsubsection{Cluster algebras for two expression of a braid group element.}
For any  expression \( \mathbf{i} \) of \( b \) with length \( \ell \), the Lusztig parameter of \( D_{\mathbf{i},s} \) is \( \mathbf{i}[s, \ell\} \). In the matrix \( B_{\bfi} \), positive entries \( b_{i_t, i_s} > 0 \) are classified as:

1. \emph{Horizontal entry}: \( s^+ \),
2. \emph{Vertical entries}: \( t \) with \( t^- < s^- < t < s \) and \( c_{i_s i_t} \neq 0 \).

Define:

\[
D_1 := D_{\mathbf{i},s^+} \prod_{t^- < s^- < t < s} D_{\mathbf{i},t} ^{-c_{i_t i_s}}.
\]

Similarly, negative entries \( b_{i_t, i_s} < 0 \) are classified as:

1. \emph{Horizontal entry}: \( s^- \),
2. \emph{Vertical entries}: \( t \) with \( s^- < t^- < s < t \) and \( c_{i_s i_t} \neq 0 \).

Define:
\[
D_2 := D_{\mathbf{i},s^-} \prod_{s^- < t^- < s < t} D_{\mathbf{i},t}^{-c_{i_s i_t}}.
\]

If Conjecture \ref{con_boscluster} holds for the word $\bfi$, the cluster structure for the seed $\bt_{\bfi}$ on \( \widehat{\cA}(b) \) implies:
\begin{equation}\label{eq_prodDs}
D_{\mathbf{i},s} \mu_s(D_{\mathbf{i},s}) = q^A D_1 + q^B D_2 \text{ for some number $A,B$}.
\end{equation}

\begin{lemma}\label{lem_mutationlusz}
Suppose Conjecture~\ref{con_boscluster} holds for the expression
\(\mathbf i\). Let \(s\in K^{\mathrm{ex}}\), and write the exchange
relation at \(s\) as
\[
D_{\mathbf i,s}\,\mu_s(D_{\mathbf i,s})
=
q^A D_1+q^B D_2,
\]
where \(D_1\) and \(D_2\) are the two exchange monomials. Then
\[
\mathcal L_{\mathbf i}(\mu_s(D_{\mathbf i,s}))
=
\mathcal L_{\mathbf i}(D_2)-\mathbf i[s,\ell\}.
\]
\end{lemma}

\begin{proof}
Since Conjecture~\ref{con_boscluster} holds for \(\mathbf i\), the mutated
variable \(\mu_s(D_{\mathbf i,s})\) is a global basis element. Moreover,
the exchange monomials \(D_1\) and \(D_2\) are cluster monomials, hence
also belong to the global basis. Note that we have
\[
\mathcal L_{\mathbf i}(D_{\mathbf i,s})
=
\mathbf i[s,\ell\}.
\]
Applying Lemma~\ref{lem_prodglobal} to the exchange relation, we obtain
\[
\mathcal L_{\mathbf i}(\mu_s(D_{\mathbf i,s}))
+
\mathbf i[s,\ell\}
=
\max\{
\mathcal L_{\mathbf i}(D_1),
\mathcal L_{\mathbf i}(D_2)
\}.
\]

By the description of the two exchange monomials, the minimal index of a
nonzero coordinate in \(\mathcal L_{\mathbf i}(D_1)\) is strictly greater
than \(s^-\), whereas the minimal index of a nonzero coordinate in
\(\mathcal L_{\mathbf i}(D_2)\) is exactly \(s^-\). Therefore, with respect
to the order on \(\mathbb Z_{\geq 0}^{[1,\ell]}\) in
Definiton~\ref{def_order}, we have
\[
\mathcal L_{\mathbf i}(D_1)\prec
\mathcal L_{\mathbf i}(D_2).
\]
Hence
\[
\max\{
\mathcal L_{\mathbf i}(D_1),
\mathcal L_{\mathbf i}(D_2)
\}
=
\mathcal L_{\mathbf i}(D_2).
\]
Thus
\[
\mathcal L_{\mathbf i}(\mu_s(D_{\mathbf i,s}))
+
\mathbf i[s,\ell\}
=
\mathcal L_{\mathbf i}(D_2),
\]
and therefore
\[
\mathcal L_{\mathbf i}(\mu_s(D_{\mathbf i,s}))
=
\mathcal L_{\mathbf i}(D_2)-\mathbf i[s,\ell\}.
\]
\end{proof}

\begin{proposition}
\label{pro_mus1}
Let \(\bfi=(i_1,\dots,i_\ell)\) be a word such that
\[
(i_{s-1},i_s,i_{s+1})=(i,j,i),
\qquad
c_{ij}c_{ji}=1.
\]
Assume that Conjecture~\ref{con_boscluster} holds for \(\bfi\). Let
\(\bfi'\) be obtained from \(\bfi\) by replacing the subword
\[
(i,j,i)
\]
in positions \(s-1,s,s+1\) with
\[
(j,i,j).
\]
Then
\[
D_{\bfi',s+1}=\mu_{s+1}(D_{\bfi,s+1}),
\qquad
D_{\bfi',s-1}=D_{\bfi,s},
\qquad
D_{\bfi',s}=D_{\bfi,s-1}.
\]
\end{proposition}

\begin{proof}
We first consider \(D_{\bfi',s+1}\). Since
\[
(s+1)^-=s-1,
\]
the vertex \(s+1\) is mutable. The exchange relation at \(s+1\) has two
exchange monomials
\[
D_1=D_{\bfi,(s+1)^+}D_{\bfi,s},
\qquad
D_2=D_{\bfi,s-1}D_{\bfi,s^+}.
\]

By the assumption that Conjecture~\ref{con_boscluster} holds for \(\bfi\),
the mutated variable
\[
\mu_{s+1}(D_{\bfi,s+1})
\]
is a global basis element. Hence Lemma~\ref{lem_mutationlusz} applies.
Using the above description of \(D_2\), we obtain
\[
\mathcal L_{\bfi}\bigl(\mu_{s+1}(D_{\bfi,s+1})\bigr)
=
(1_{s-1},0,0,\bfi[s^+,\ell\}).
\]
Here the displayed vector is written with respect to the local block
\((s-1,s,s+1)\), followed by the unchanged tail.

By Theorem~\ref{theo_braidmoves}, the braid-move transition map sends the
\(\bfi'\)-Lusztig datum
\[
\bfi'[s+1,\ell\}
\]
to the \(\bfi\)-Lusztig datum
\[
(1_{s-1},0,0,\bfi[s^+,\ell\}).
\]
That is,
\[
\Phi_{\bfi'}^{\bfi}\bigl(\bfi'[s+1,\ell\}\bigr)
=
(1_{s-1},0,0,\bfi[s^+,\ell\}).
\]
Therefore the \(\bfi'\)-Lusztig datum of
\(\mu_{s+1}(D_{\bfi,s+1})\) is precisely
\[
\bfi'[s+1,\ell\}.
\]
Since global basis elements are uniquely determined by their Lusztig
parameters, we get
\[
\mu_{s+1}(D_{\bfi,s+1})=D_{\bfi',s+1}.
\]

Next, we prove
\[
D_{\bfi',s-1}=D_{\bfi,s}.
\]
By definition,
\[
D_{\bfi,s}=B(\bfi,\bfi[s,\ell\}),
\qquad
D_{\bfi',s-1}=B(\bfi',\bfi'[s-1,\ell\}).
\]
In the local block, the \(\bfi\)-Lusztig datum of \(D_{\bfi,s}\) is
\[
(0,1,0),
\]
whereas the \(\bfi'\)-Lusztig datum of \(D_{\bfi',s-1}\) is
\[
(1,0,1).
\]
Outside the local block \((s-1,s,s+1)\), the two data agree. By
Theorem~\ref{theo_braidmoves}, the braid-move transition map sends
\[
(0,1,0)
\longmapsto
(1,0,1).
\]
Equivalently,
\[
\Phi_{\bfi}^{\bfi'}(\bfi[s,\ell\})
=
\bfi'[s-1,\ell\}.
\]
Hence
\[
D_{\bfi',s-1}=D_{\bfi,s}.
\]

Finally, we prove
\[
D_{\bfi',s}=D_{\bfi,s-1}.
\]
Indeed,
\[
D_{\bfi,s-1}=B(\bfi,\bfi[s-1,\ell\}),
\qquad
D_{\bfi',s}=B(\bfi',\bfi'[s,\ell\}).
\]
In the local block, the corresponding Lusztig data are
\[
(1,0,1)
\qquad\text{and}\qquad
(0,1,0),
\]
respectively, and outside the local block they agree. Again by
Theorem~\ref{theo_braidmoves},
\[
\Phi_{\bfi}^{\bfi'}(\bfi[s-1,\ell\})
=
\bfi'[s,\ell\}.
\]
Therefore
\[
D_{\bfi',s}=D_{\bfi,s-1}.
\]
\end{proof}

The following theorem demonstrates that if Conjecture \ref{con_boscluster} holds for one expression of \( b \), it extends to all expressions.

\begin{theorem}
\label{theo_words}
Let \(C\) be a Cartan matrix with
\[
c_{ij}c_{ji}\leq 1
\qquad (i\neq j).
\]
If Conjecture~\ref{con_boscluster} holds for an expression
\(\bfi\) of \(b\in \Br^+\), then it holds for every expression
\(\bfi'\) of \(b\).
\end{theorem}

\begin{proof}
Since \(c_{ij}c_{ji}\leq 1\), any two expressions of \(b\in\Br^+\) are
connected by a finite sequence of \(2\)-moves and \(3\)-moves. Hence it
suffices to prove that Conjecture~\ref{con_boscluster} is preserved under
one such move.

Assume first that \(\bfi'\) is obtained from \(\bfi\) by a \(2\)-move
interchanging the adjacent letters \(i_k\) and \(i_{k+1}\), where
\[
c_{i_k i_{k+1}}=0.
\]
Let \(\sigma_k\) denote the transposition of the positions \(k\) and
\(k+1\). By Theorem~\ref{theo_braidmoves}, we have
\[
D_{\bfi',s}=D_{\bfi,\sigma_k(s)}
\qquad
(s\in[1,\ell]).
\]
Moreover, by \cite[Lemma~2.1]{fujita2023isomorphisms},
\[
B_{\bfi'}
=
(b_{\sigma_k(r),\sigma_k(t)})_{r,t},
\qquad
\Lambda_{\bfi'}
=
(\lambda_{\sigma_k(r),\sigma_k(t)})_{r,t}.
\]
The frozen set is also carried to the frozen set by the same permutation.
Therefore, the seed \(\mathbf t_{\bfi'}\) is obtained from
\(\mathbf t_{\bfi}\) by the seed isomorphism induced by \(\sigma_k\).
Hence
\[
\mathcal A(\mathbf t_{\bfi'})
=
\mathcal A(\mathbf t_{\bfi}).
\]
Since Conjecture~\ref{con_boscluster} holds for \(\bfi\), we have
\[
\mathcal A(\mathbf t_{\bfi})=\widehat{\mathcal A}(b),
\]
and all cluster monomials of \(\mathcal A(\mathbf t_{\bfi})\) belong to
the global basis of \(\widehat{\mathcal A}(b)\). The same conclusions
therefore hold for \(\mathbf t_{\bfi'}\).

Next assume that \(\bfi'\) is obtained from \(\bfi\) by a \(3\)-move
replacing
\[
(i,j,i)
\]
with
\[
(j,i,j)
\]
at the positions \(k-1,k,k+1\), where
\[
c_{ij}c_{ji}=1.
\]
By Proposition~\ref{pro_mus1}, the cluster variables in the local block
satisfy
\[
D_{\bfi',k+1}
=
\mu_{k+1}(D_{\bfi,k+1}),
\qquad
D_{\bfi',k-1}
=
D_{\bfi,k},
\qquad
D_{\bfi',k}
=
D_{\bfi,k-1}.
\]
For \(s\notin\{k-1,k,k+1\}\), the braid move does not affect the
corresponding Lusztig datum, and hence
\[
D_{\bfi',s}=D_{\bfi,s}.
\]
Thus the cluster variables of \(\mathbf t_{\bfi'}\) are obtained from the
cluster variables of \(\mathbf t_{\bfi}\) by first mutating at \(k+1\) and
then applying the permutation \(\sigma_{k-1}\), which interchanges the
positions \(k-1\) and \(k\).

Furthermore, by \cite[Lemmas~2.7 and~2.8]{fujita2023isomorphisms}, the
exchange matrix and the compatible form transform in the same way:
\[
B_{\bfi'}
=
\sigma_{k-1}\mu_{k+1}(B_{\bfi}),
\qquad
\Lambda_{\bfi'}
=
\sigma_{k-1}\mu_{k+1}(\Lambda_{\bfi}).
\]
The frozen set is also identified under this mutation followed by the
permutation \(\sigma_{k-1}\). Therefore
\[
\mathbf t_{\bfi'}
=
\sigma_{k-1}\mu_{k+1}(\mathbf t_{\bfi})
\]
as quantum seeds. In particular, the two seeds are mutation equivalent,
and hence
\[
\mathcal A(\mathbf t_{\bfi'})
=
\mathcal A(\mathbf t_{\bfi})
=
\widehat{\mathcal A}(b).
\]

Since the cluster monomials of \(\mathcal A(\mathbf t_{\bfi})\) belong to
the global basis by assumption, and since \(\mathbf t_{\bfi'}\) is obtained
from \(\mathbf t_{\bfi}\) by a mutation followed by a seed isomorphism, the
cluster monomials for \(\mathbf t_{\bfi'}\) are the same cluster monomials
viewed in the same cluster algebra. Hence they also belong to the global
basis of \(\widehat{\mathcal A}(b)\).

Thus Conjecture~\ref{con_boscluster} is preserved under both \(2\)-moves
and \(3\)-moves. Since any two expressions of \(b\) are connected by such
moves, the theorem follows.
\end{proof}

\subsection{A sequence of mutations}
In this section, we introduce a sequence of mutations and study its effect
on the quiver $Q_{\bfi}$. The resulting description of the mutated quivers
will lead to the quantum $T$-system for the quantum minors in
$\widehat{\cA}(b)$.

For $k\in[1,\ell]$, recall that $n_{i_k}$ denotes the total number of
vertices of color $i_k$ in the word $\bfi$. For $j\in I$, define
\[
k[j]:=\#\{\,s\in[k,\ell]\mid i_s=j\,\}.
\]
Thus $k[j]$ counts the number of vertices of color $j$ lying weakly to the right of $k$ in the word $\bfi$.

\begin{definition}
For $k\in K^{\mathrm{ex}}$, Let
\[r_k=n_{i_k}-1-k[i_k]\]

define
\[
\widetilde{\mu}_k
:=
\mu_{k_{\max}^{r_{i_k}}}\cdots \mu_{k_{\max}^{1}}\mu_{k_{\max}},
\]
where the mutations are taken successively along the chain of predecessors
of $k$. For $l\in [1,\ell]$, set
\[
M_l
:=
\widetilde{\mu}_l\widetilde{\mu}_{l+1}\cdots\widetilde{\mu}_\ell .
\] Here we set $\widetilde\mu_l=\Id$ if $l\in K^{\operatorname{fr}}$. For $k\in[1,\ell]$, define
\[
J_k:=\{(i,p)\mid i\in I,\ 1\le p\le k[i]\}.
\]
Here we use the occurrence notation for vertices introduced above.
We define $\widetilde Q_k$ to be the full subquiver of $M_k(Q_{\bfi})$
obtained by deleting the vertices in $J_k$; equivalently,
$\widetilde Q_k$ is the induced subquiver of $M_k(Q_{\bfi})$ on the
vertex set
\[
\operatorname{Vert}(M_k(Q_{\bfi}))\setminus J_k.
\]
We declare a vertex of $\widetilde Q_k$ to be frozen if it belongs in $\{(i,k[i]+1)\mid i\in I\}$. We denote by $\widehat Q_k$ the quiver obtained from $\widetilde Q_k$ by deleting all arrows between frozen vertices.
\end{definition}

In occurrence notation, the mutation sequence associated with \(k\) is
taken along the vertices of color \(i_k\), from right to left:
\[
(i_k,1+k[i_k])
\longleftarrow
\cdots
\longleftarrow
(i_k,n_{i_k}-1)
\longleftarrow
(i_k,n_{i_k})
\]
\begin{example}\label{exam_bfi}
    Let us consider type $A_3$ and a word $\bfi=(122313213)$. The quiver $Q_\bfi$ is given by Figure \ref{fig:compact-quiver}.
   \begin{figure}
\begin{tikzpicture}[
    x=1.0cm,
    y=1.0cm,
    >={Stealth[length=1.8mm]},
    every node/.style={
        circle,
        draw,
        minimum size=6mm,
        inner sep=0pt,
        font=\small
    },
    every edge/.style={draw, ->, thick}
]

\node (9) at (-0.4,2.4) {1};
\node (8) at (2.8,2.4) {5};
\node (7) at (4.8,2.4) {8};

\node (6) at (0.4,1.2) {2};
\node (5) at (1.6,1.2) {3};
\node (4) at (4.0,1.2) {7};

\node (3) at (2.0,0) {4};
\node (2) at (3.2,0) {6};
\node (1) at (5.2,0) {9};

\draw (8) edge (9);
\draw (7) edge (8);
\draw (6) edge (8);
\draw (8) edge (4);
\draw (4) edge (7);

\draw (5) edge (6);
\draw (4) edge (5);

\draw (2) edge (3);
\draw (1) edge (2);

\draw (3) edge (4);
\draw (4) edge (1);

\end{tikzpicture}
\caption{Quiver $Q_\bfi$}
\label{fig:compact-quiver}
\end{figure}
We have
\[
\widetilde{\mu}_9=\mu_6\mu_9,\quad
\widetilde{\mu}_8=\mu_5\mu_8,\quad
\widetilde{\mu}_7=\mu_3\mu_7,\quad
\widetilde{\mu}_6=\mu_9 ,\quad \cdots
\]
To determine \(J_6\), we first compute
\[
6[1]=1,\qquad 6[2]=1,\qquad 6[3]=2.
\]
It follows, from the definition of \(J_6\), that
\[
J_6=\{1,2,4,6\}.
\]
\end{example}
\noindent

\subsubsection{The two-color subquiver associated with adjacent vertices}
Let \(p,q\in[1,\ell]\) be two vertices such that
\[
c_{i_p i_q}c_{i_q i_p}=1.
\]
We denote by \(Q_{p,q}\) the full subquiver of \(Q_{\bfi}\) whose vertices
have color either \(i_p\) or \(i_q\).

We now describe the ordinary arrows in \(Q_{p,q}\). Let
\begin{equation}\label{eq_first_arrow}
j^1\longrightarrow k^1
\end{equation}
be the rightmost ordinary arrow in \(Q_{p,q}\).
Here “rightmost” means that \(j^1\) and \(k^1\) are chosen maximal among
vertices of colors \(i_p\) and \(i_q\), subject to the
condition that there is an ordinary arrow between them. Suppose
\[
i_{j^1}=i_p,
\qquad
i_{k^1}=i_q.
\]

We define two sequences of vertices \(\{j^n\}\) and \(\{k^n\}\)
inductively. Suppose that \(j^{n-1}\) has been defined. Using the notation
introduced in \eqref{eq:notationa+a-}, set
\[
k^n:=(j^{n-1})^-(i_q)^+,
\qquad
j^n:=(k^n)^-(i_p)^+, \text{ for all }n\geq 2
\]
whenever these vertices exist. For instance, if $s^-=-\infty$, then we say $s^-(i)^+$ doesn't exist for any $i\in I$. In other words, \(k^n\) is obtained by
first taking the predecessor of \(j^{n-1}\), and then taking the next
vertex of color \(i_q\); similarly, \(j^n\) is obtained by first taking
the predecessor of \(k^n\), and then taking the next vertex of color
\(i_p\). Thus each \(j^n\) has color \(i_p\), and each \(k^n\) has color
\(i_q\).
\begin{lemma}\label{lem_arrows}
Assume that the rightmost ordinary arrow in $Q_{p,q}$ is \eqref{eq_first_arrow}.
Then the ordinary arrows in the two-color subquiver $Q_{p,q}$ are precisely
\[
j^n\longrightarrow k^n\qquad(n\geq 1)
\]
and
\[
k^n\longrightarrow j^{n-1}\qquad(n\geq 2),
\]
for all \(n\) for which the corresponding vertices are defined.
\end{lemma}

\begin{proof}
Recall that an ordinary arrow $u\to v$ in $Q_{\bfi}$ is characterized by
\[
u^-<v^-<u<v
\qquad\text{and}\qquad
c_{i_u i_v}=-1.
\]

We first check that the arrows listed in the statement do occur. By the
definition of $j^1$, we have
\[
(j^1)^-<(k^1)^-<j^1<k^1,
\]
hence $j^1\to k^1$ is an ordinary arrow.

Assume that $j^{n-1}\to k^{n-1}$ is an ordinary arrow. Then
\[
(j^{n-1})^-<(k^{n-1})^-<j^{n-1}<k^{n-1}.
\]
Since $k^n=(j^{n-1})^-(i_q)^+$, we have
\[
(k^n)^-<(j^{n-1})^-<k^n<j^{n-1},
\]
and hence $k^n\to j^{n-1}$ is an ordinary arrow. Similarly, since
$j^n=(k^n)^-(i_p)^+$, we get
\[
(j^n)^-<(k^n)^-<j^n<k^n,
\]
so $j^n\to k^n$ is also an ordinary arrow.

It remains to prove that there are no other ordinary arrows. Let
$a\to b$ be an ordinary arrow in $Q_{p,q}$.

First suppose that $a$ has color $i_p$ and $b$ has color $i_q$. Choose
$t$ such that
\[
j^{t-1}>a\ge j^t .
\]
If no such \(t\) exists, then \(a\) lies to the left of the last defined
\(j^t\). In this case the relevant predecessor of color \(i_q\) does not
exist, and the inequalities defining an ordinary arrow cannot be satisfied.
Thus no ordinary arrow can occur there.

If $a\neq j^t$, then $j^{t-1}>a>j^t$, and hence
\[
a^-\ge j^t,
\qquad
(j^{t-1})^-\ge a.
\]
We compare $b$ with $k^t$.

If $b=k^t$, then the ordinary-arrow inequalities give
\[
a^-<(k^t)^-<a<k^t.
\]
Since $j^t=(k^t)^-(i_p)^+$, this forces $a=j^t$, contradicting
$a>j^t$. If $b<k^t$, then $b^-\le (k^t)^-$, and therefore
\[
b^-<j^t\le a^-,
\]
contradicting $a^-<b^-$. If $b>k^t$, then $b^-\ge k^t$, and since
$k^t=(j^{t-1})^-(i_q)^+$, we get
\[
b^-\ge k^t>(j^{t-1})^-\ge a,
\]
contradicting $b^-<a$. Thus necessarily $a=j^t$.

Now, with $a=j^t$, the same comparison shows that $b$ must be $k^t$.
Indeed, if $b<k^t$, then
\[
b\le (k^t)^-<(k^t)^-(i_p)^+=j^t=a,
\]
contradicting $a<b$; while if $b>k^t$, then
\[
k^t\le b^-<a=j^t,
\]
contradicting the ordinary arrow $j^t\to k^t$. Hence the only arrows from
color $i_p$ to color $i_q$ are
\[
j^t\to k^t.
\]

Now suppose that $a$ has color $i_q$ and $b$ has color $i_p$. Choose
$t$ such that
\[
k^{t-1}>a\ge k^t .
\]
If no such \(t\) exists, then \(a\) lies to the left of the last defined
\(k^t\). In this case the relevant predecessor of color \(i_p\) does not
exist, and the inequalities defining an ordinary arrow cannot be satisfied.
Thus no ordinary arrow can occur there.
If $a\neq k^t$, then $k^{t-1}>a>k^t$, and hence
\[
a^-\ge k^t,
\qquad
(k^{t-1})^-\ge a.
\]
We compare $b$ with $j^{t-1}$.

If $b=j^{t-1}$, then, since
$k^t=(j^{t-1})^-(i_q)^+$, we have
\[
b^-=(j^{t-1})^-<k^t\le a^-,
\]
contradicting $a^-<b^-$. If $b>j^{t-1}$, then
\[
b^-\ge j^{t-1}=(k^{t-1})^-(i_p)^+>(k^{t-1})^-\ge a,
\]
contradicting $b^-<a$. If $b<j^{t-1}$, then
\[
b^-<(j^{t-1})^-<k^t\le a^-,
\]
again contradicting $a^-<b^-$. Therefore $a=k^t$.

Finally, with $a=k^t$, the vertex $b$ must be $j^{t-1}$. If
$b<j^{t-1}$, then
\[
b\le (j^{t-1})^-<(j^{t-1})^-(i_q)^+=k^t=a,
\]
contradicting $a<b$. If $b>j^{t-1}$, then
\[
j^{t-1}\le b^-<a=k^t,
\]
contradicting the ordinary arrow $k^t\to j^{t-1}$. Hence the only arrows
from color $i_q$ to color $i_p$ are
\[
k^t\to j^{t-1}.
\]

Combining the two orientations, the ordinary arrows in $Q_{p,q}$ are
exactly
\[
j^t\to k^t
\qquad\text{and}\qquad
k^t\to j^{t-1}.
\]
\end{proof}

The following figure illustrates the pattern of ordinary arrows in the
two-color subquiver $Q_{p,q}$, where the vertices $j^m$ have color
$i_p$ and the vertices $k^m$ have color $i_q$.

\begin{figure}[htbp]
\centering
\begin{tikzpicture}[
  x=1.2cm,
  y=0.9cm,
  >={Stealth[length=1.7mm]},
  every node/.style={font=\small},
  arr/.style={->, thick, shorten >=1pt, shorten <=1pt}
]

\node (jtp1) at (0,0) {$j^{t+1}$};
\node (ktp1) at (1,1) {$k^{t+1}$};

\node (jt)   at (2.8,0) {$j^t$};
\node (kt)   at (3.8,1) {$k^t$};

\node (dotsbot) at (5.2,0) {$\cdots$};
\node (dotstop) at (6.0,1) {$\cdots$};

\node (j2) at (7.8,0) {$j^2$};
\node (k2) at (8.8,1) {$k^2$};

\node (j1) at (10.6,0) {$j^1$};
\node (k1) at (11.6,1) {$k^1$};

\draw[arr] (jtp1) -- (ktp1);
\draw[arr] (ktp1) -- (jt);

\draw[arr] (jt) -- (kt);
\draw[arr] (kt) -- (dotsbot);

\draw[arr] (dotstop) -- (j2);
\draw[arr] (j2) -- (k2);

\draw[arr] (k2) -- (j1);
\draw[arr] (j1) -- (k1);

\end{tikzpicture}
\caption{Ordinary arrows in the two-color subquiver $Q_{p,q}$.}
\label{fig:ordinary-arrows-pq}
\end{figure}

\begin{lemma}\label{lem_first_ordinary}
Let \(Q_{\bfi}\) be the quiver associated with the exchange matrix
\(B_{\bfi}\). Let \(p\in[1,\ell]\) be such that
\[
c_{i_p i_\ell}c_{i_\ell i_p}=1.
\]
Assume that the two-color subquiver \(Q_{p,\ell}\) contains an ordinary
arrow. Then its rightmost ordinary arrow is oriented from the vertex of
color \(i_p\) to the vertex of color \(i_\ell\).
\end{lemma}
\begin{proof}
Assume, for a contradiction, that the rightmost ordinary arrow in
\(Q_{p,\ell}\) is of the form
\[
\ell'\longrightarrow p',
\]
where
\[
i_{\ell'}=i_\ell,
\qquad
i_{p'}=i_p.
\]
By the defining rule for ordinary arrows, we have
\[
(\ell')^-<(p')^-<\ell'<p'.
\]
Since \(p'<\ell\) and \(i_\ell=i_{\ell'}\), there exists a first vertex of
color \(i_\ell\) lying to the right of \(p'\). Denote it by
\[
\ell'':=p'(i_\ell)^+.
\]
Then \(\ell''\leq \ell\). Moreover, since \((\ell'')^-\) is the previous
vertex of color \(i_\ell\) before \(\ell''\), and since \(\ell'\) is a
vertex of color \(i_\ell\) lying before \(p'\), we have
\[
\ell'\leq (\ell'')^-<p'<\ell''.
\]
Together with
\[
(p')^-<\ell',
\]
this gives
\[
(p')^-<(\ell'')^-<p'<\ell''.
\]
Since \(c_{i_p i_\ell}=-1\), the defining rule for ordinary arrows gives
an ordinary arrow
\[
p'\longrightarrow \ell''.
\]
But this ordinary arrow lies strictly to the right of
\(\ell'\to p'\) in the color \(i_\ell\)-coordinate, while it has the same
vertex \(p'\) of color \(i_p\). This contradicts the choice of
\(\ell'\to p'\) as the rightmost ordinary arrow.

Therefore the rightmost ordinary arrow cannot be oriented from color
\(i_\ell\) to color \(i_p\). Hence it is oriented from the vertex of color
\(i_p\) to the vertex of color \(i_\ell\).
\end{proof}

\subsubsection{The effect of the mutation sequence $\widetilde\mu_{\ell}$ on $Q_{\bfi}$}
To analyze the effect of the mutation sequence \(\widetilde{\mu}_{\ell}\)
on the quiver \(Q_{\bfi}\), we restrict to the two-color full subquiver
\(Q_{p,\ell}\), where
\[
c_{i_p i_\ell}c_{i_\ell i_p}=1.
\]
If \(Q_{p,\ell}\) contains no ordinary arrow, then there is nothing to
prove for this pair of colors. Thus we assume that ordinary arrows occur.

By Lemma~\ref{lem_first_ordinary}, the rightmost ordinary arrow in
\(Q_{p,\ell}\) is oriented from color \(i_p\) to color \(i_\ell\). Hence,
using the notation of Lemma~\ref{lem_arrows}, we may label the vertices so
that the vertices \(j^t\) have color \(i_p\), while the vertices \(k^t\)
have color \(i_\ell\). Lemma~\ref{lem_arrows} then shows that the
ordinary arrows in \(Q_{p,\ell}\) form a zig-zag chain
\[
j^t\longrightarrow k^t,
\qquad
k^{t+1}\longrightarrow j^t,
\]
for all \(t\) for which the corresponding vertices are defined.

Our goal is to understand how the mutation sequence
\(\widetilde{\mu}_{\ell}\) changes this two-color subquiver. Since
\(\widetilde{\mu}_{\ell}\) mutates successively at vertices of color
\(i_\ell\), the effect on \(Q_{p,\ell}\) can be analyzed locally around
each segment of the zig-zag chain
\[
k^{t+1}\longrightarrow j^t\longrightarrow k^t.
\]
Indeed, when one mutates at a vertex of color \(i_\ell\), the only arrows
inside the two-color subquiver that can change are the arrows incident to
that vertex, together with the arrows created or cancelled by length-two
paths through that vertex. Thus the relevant local data are the three
consecutive vertices
\[
k^{t+1},\qquad k^t,\qquad (k^t)^+,
\]
together with the intervening vertex \(j^t\) of color \(i_p\).\\

\noindent
The four possible local configurations appearing during this process are
shown in Figure~\ref{fig:four-local-configurations}.
\begin{figure}[htbp]
\centering
\begin{tikzpicture}[
   x=1.1cm,
  y=1.0cm,
  >={Stealth[length=1.8mm]},
  every node/.style={font=\small},
  arr/.style={->, thick, shorten >=1pt, shorten <=1pt},
  darr/.style={->, thick, dashed, shorten >=1pt, shorten <=1pt},
  lab/.style={font=\small\bfseries}
]

\node (q1) at (0,0) {$k^{t+1}$};
\node (c1) at (3,0) {$k^{t}$};
\node (u1) at (4.5,0) {$(k^t)^+$};
\node (p1) at (2,1.7) {$j^t$};

\draw[arr] (q1) -- (p1);
\draw[arr] (p1) -- (c1);
\draw[darr] (c1) -- (q1);
\draw[arr] (c1) -- (u1);

\node[lab] at (3,-0.8) {(1)};

\node (q2) at (6,0) {$k^{t+1}$};
\node (c2) at (8.5,0) {$k'$};
\node (u2) at (10,0) {$(k')^+$};
\node (w2) at (12,0) {$(k^t)^+$};
\node (p2) at (9,1.7) {$j^t$};

\draw[arr] (q2) -- (p2);
\draw[arr] (p2) -- (c2);
\draw[arr] (u2) -- (p2);
\draw[darr] (c2) -- (q2);
\draw[arr] (c2) -- (u2);
\draw[darr] (u2) -- (w2);
\draw[arr] (p2) -- (w2);

\node[lab] at (10,-0.8) {(2)};

\node (q3) at (1,-3.2) {$k^{t+1}$};
\node (c3) at (2.5,-3.2) {$(k^{t+1})^+$};
\node (u3) at (5,-3.2) {$(k^t)^+$};
\node (p3) at (4,-1.5) {$j^t$};
\node (v3) at (0,-1.5) {$j^{t+1}$};

\draw[arr] (p3) -- (u3);
\draw[arr] (q3) -- (c3);
\draw[arr] (c3) -- (p3);
\draw[darr] (u3) -- (c3);
\draw[arr] (v3) -- (q3);

\node[lab] at (3,-4.0) {(3)};

\node (q4) at (7,-3.2) {$(k^{t+2})^+$};
\node (c4) at (9.5,-3.2) {$(k^{t+1})^+$};
\node (u4) at (12,-3.2) {$(k^t)^+$};
\node (p4) at (8.5,-1.5) {$j^{t+1}$};
\node (r4) at (11,-1.5) {$j^t$};

\draw[arr] (p4) -- (q4);
\draw[arr] (p4) -- (c4);
\draw[darr] (c4) -- (q4);
\draw[arr] (c4) -- (r4);
\draw[arr] (r4) -- (u4);
\draw[darr] (u4) -- (c4);

\node[lab] at (10,-4.0) {(4)};

\end{tikzpicture}
\caption{Four possible local configurations arising in the mutation
process. Dashed arrows denote paths consisting of horizontal arrows.}
\label{fig:four-local-configurations}
\end{figure}

\smallskip

\noindent
\textbf{Local mutation rule.}
We shall repeatedly use the usual quiver mutation rule. When mutating at a
vertex \(v\), one creates an arrow \(x\to y\) for each oriented path
\(x\to v\to y\), then reverses all arrows incident with \(v\), and finally
cancels all oriented 2-cycles.

\smallskip

\noindent
\textbf{Step 1: the initial local configuration.}
We first consider the rightmost segment of the zig-zag chain. By the choice
of the rightmost ordinary arrow, no ordinary arrow of \(Q_{p,\ell}\) is
incident with a vertex of color \(i_\ell\) lying strictly to the right of
\(k^1\). Hence the mutations
\[
\mu_{(k^1)^+}\cdots \mu_\ell
\]
affect only the horizontal arrows joining consecutive vertices of color
\(i_\ell\) inside the two-color subquiver. In particular, the ordinary
arrows in the segment
\[
k^{2}\longrightarrow j^1\longrightarrow k^1
\]
remain unchanged. Thus the local configuration around
\[
k^2,\qquad j^1,\qquad k^1,\qquad (k^1)^+
\]
is the one shown in Figure~\ref{fig:four-local-configurations}\,(1). If \(k^2\) does not exist, then this rightmost segment consists only of the
last arrow, and we may pass directly to the case considered in \eqref{eq_last_arrows}.

\smallskip

\noindent
\textbf{Step 2: mutating between \(k^t\) and \((k^{t+1})^+\).}
We next mutate successively at the vertices of color \(i_\ell\) lying
between \(k^t\) and \((k^{t+1})^+\). Let
\[
v=(k')^+
\]
be one of these vertices. Mutating at \(v\) creates arrows along all
oriented paths through \(v\), reverses all arrows incident with \(v\), and
cancels the resulting oriented 2-cycles. After each such mutation, the same local pattern is shifted one step to the
left; a typical intermediate configuration is shown in
Figure~\ref{fig:four-local-configurations}\,(2).
\smallskip

\noindent
\textbf{Step 3: mutating at \((k^{t+1})^+\).}
We now mutate at \((k^{t+1})^+\). Indeed, the oriented path
\[
j^t\longrightarrow (k^{t+1})^+\longrightarrow k^{t+1}
\]
creates an arrow
\[
j^t\longrightarrow k^{t+1},
\]
which is opposite to the existing arrow
\[
k^{t+1}\longrightarrow j^t.
\]
Hence these two arrows form a 2-cycle and cancel. After reversing the
remaining arrows incident with \((k^{t+1})^+\), the local configuration is
exactly the one shown in Figure~\ref{fig:four-local-configurations}\,(3).

\smallskip

\noindent
\textbf{Step 4: propagation to the next segment.}
Repeating the same argument one step further to the left, the role of the
pair \((j^t,k^t)\) is replaced by the next pair
\((j^{t+1},k^{t+1})\). Thus the local configuration becomes the one shown
in Figure~\ref{fig:four-local-configurations}\,(4). In this way, the same
pattern propagates inductively along the chain of vertices of color
\(i_\ell\).

At the boundary, if \((k^1)^-\) does not exist, then no arrow involving
\((k^1)^-\) appears.

Assume that
\begin{equation}\label{eq_last_arrows}
j^{t+1}\longrightarrow k^{t+1}
\end{equation}
is the last ordinary arrow in \(Q_{p,\ell}\). We claim that \(j^{t+1}\) is
a boundary vertex of the induced subquiver, and hence is frozen. Indeed,
if \((j^{t+1})^-\) existed, then \(k^{t+2}\) would be defined. By the
construction of the zig-zag chain, this would yield another ordinary arrow
\[
k^{t+2}\longrightarrow j^{t+1},
\]
contradicting the assumption that \(j^{t+1}\to k^{t+1}\) is the last
ordinary arrow. Thus \(j^{t+1}\) is a boundary vertex, hence frozen.

In this case the final local configuration gives
\begin{equation}\label{eq_frozen_arrows}
(\ell_{\min})^+\longrightarrow j^{t+1}\longrightarrow \ell_{\min}.
\end{equation}

Summarizing, the four configurations in
Figure~\ref{fig:four-local-configurations} describe the successive local
forms of the subquiver \(Q_{p,\ell}\) during the mutation sequence
\(\widetilde{\mu}_{\ell}\). This yields the following lemma.

\begin{lemma}
\label{lem:embedding-Qell}
During the mutation sequence \(\widetilde{\mu}_{\ell}\), let \((i_\ell,k)\) be one of the vertices of color \(i_\ell\) appearing in
the mutation sequence \(\widetilde\mu_\ell\). Immediately before mutating at \((i_\ell,k)\), all arrows
with source \((i_\ell,k)\) are horizontal arrows, while all arrows with
target \((i_\ell,k)\) are ordinary arrows.

Moreover, there is an isomorphism of quivers
\[
\Phi_\ell:\widehat Q_\ell\simeq Q_{\bfi}\setminus\{\ell\}.
\]
Under this isomorphism, the frozen vertices of \(\widehat Q_\ell\) are sent
exactly to the frozen vertices of \(Q_{\bfi}\setminus\{\ell\}\).
\end{lemma}

\begin{proof}
The first assertion follows from the local analysis above. Indeed, for
each color \(i_p\) adjacent to \(i_\ell\), that is, satisfying
\[
c_{i_p i_\ell}c_{i_\ell i_p}=1,
\]
the two-color subquiver \(Q_{p,\ell}\) evolves according to the four local
configurations in Figure~\ref{fig:four-local-configurations}. These configurations show that, during the mutation sequence
\(\widetilde\mu_\ell\), every arrow starting from the mutated vertex of
color \(i_\ell\) is horizontal, and every arrow ending at it is ordinary.
The ordinary arrows come only from the two-color subquivers
\(Q_{p,\ell}\) with adjacent colors. Hence the first assertion follows.

We now construct the isomorphism \(\Phi_\ell\). In occurrence notation,
define
\[
\Phi_\ell(i,k)=(i,k)
\qquad\text{if } i\neq i_\ell,
\]
and
\[
\Phi_\ell(i_\ell,k)=(i_\ell,k-1)
\qquad\text{for } k\geq 2.
\]
The vertex \((i_\ell,1)=\ell_{\min}\) is deleted. Hence \(\Phi_\ell\) is a bijection from the vertex set
of \(\widehat Q_\ell\) to the vertex set of \(Q_{\bfi}\setminus\{\ell\}\).

It remains to check that \(\Phi_\ell\) preserves arrows. We first consider
horizontal arrows. After applying the mutation sequence
\(\widetilde\mu_\ell\), the horizontal arrows of color \(i_\ell\) have the
form
\begin{equation}
\label{eq_horizontal-arrows}
(i_\ell,1)\longrightarrow (i_\ell,2)
\longleftarrow \cdots
\longleftarrow (i_\ell,n_{i_\ell}).
\end{equation}
The only possible extra horizontal arrow at the boundary is
\[
(i_\ell,1)\longrightarrow (i_\ell,2).
\]
Since \(\ell_{\min}\in J_\ell\), this arrow is removed when we pass to the
full subquiver \(\widehat Q_\ell\).

Under the map \(\Phi_\ell\), the remaining horizontal arrows are shifted
one step to the left and become
\[
(i_\ell,1)\longleftarrow (i_\ell,2)
\longleftarrow \cdots
\longleftarrow (i_\ell,n_{i_\ell}-1).
\]
These are exactly the horizontal arrows in
\(Q_{\bfi}\setminus\{\ell\}\). Hence \(\Phi_\ell\) preserves horizontal
arrows.

We next consider ordinary arrows. For every adjacent color \(i_p\), let us consider the subquiver $Q_{p,\ell}$. by Lemma \ref{lem_first_ordinary}, we get $i_{k^1}=i_\ell$ the $i_{j^1}=i_p$. 
The local configurations in Figure~\ref{fig:four-local-configurations} show
that all non-boundary ordinary arrows in the mutated two-color subquiver
are sent by \(\Phi_\ell\) to the corresponding ordinary arrows in
\(Q_{\bfi}\setminus\{\ell\}\). Hence it remains only to examine the left
boundary of each two-color chain.

First suppose that the last ordinary arrow in \(Q_{p,\ell}\) is
\[
j^{t+1}\longrightarrow k^{t+1}.
\]
We claim that \(j^{t+1}\) is frozen. Indeed, if \((j^{t+1})^-\) existed,
then, by setting
\[
k^{t+2}:=(j^{t+1})^-(i_\ell)^+,
\]
the defining rule for ordinary arrows would give another ordinary arrow
\[
k^{t+2}\longrightarrow j^{t+1},
\]
contradicting the assumption that \(j^{t+1}\to k^{t+1}\) is the last
ordinary arrow. Hence \((j^{t+1})^-\) does not exist, and therefore
\(j^{t+1}\) is frozen.

In this boundary case, the local mutation process produces the arrows
\[
\ell_{\min}^+\longrightarrow j^{t+1},
\quad
j^{t+1}\longrightarrow \ell_{\min}, \quad\text{and} \quad j^{t+1}\to (k^{t+1})^+
\]
The second arrow is removed when we pass to the full subquiver
\(\widehat Q_\ell\), because \(\ell_{\min}\in J_\ell\). The first arrow
is an arrow between frozen vertices: \(j^{t+1}\) is frozen by the previous
paragraph, and \(\ell_{\min}^+\) is frozen in \(\widehat Q_\ell\) by the
definition of the frozen vertices after deleting \(J_\ell\). Hence this
arrow is deleted when passing from \(\widetilde Q_\ell\) to
\(\widehat Q_\ell\). Therefore only arrow $j^{t+1}\to (k^{t+1})^+$ remains in
\(\widehat Q_\ell\) in this case.

Now suppose that the last ordinary arrow is
\[
k^{t+1}\longrightarrow j^t.
\]
Then the ordinary-arrow inequalities give
\[
(k^{t+1})^-<(j^t)^-<k^{t+1}<j^t.
\]
In particular, \((j^t)^-\) exists, so \(j^t\) is not frozen. Thus the
arrows appearing in the final local configuration are not removed. Under the shift
\[
(i_\ell,k)\mapsto(i_\ell,k-1),
\]
these arrows are sent exactly to the corresponding ordinary arrows of
\(Q_{\bfi}\setminus\{\ell\}\).
Indeed, away from the boundary, the local configurations merely shift the
vertices of color \(i_\ell\) by one occurrence, while leaving the vertices
of the other color unchanged. This is precisely the action of
\(\Phi_\ell\).

Combining the horizontal-arrow and ordinary-arrow analyses, we conclude
that \(\Phi_\ell\) preserves arrows and induces an isomorphism
\[
\Phi_\ell:\widehat Q_\ell\simeq Q_{\bfi}\setminus\{\ell\}.
\]

Finally, let us check frozen vertices. For colors \(i\neq i_\ell\),
\(\Phi_\ell\) fixes the occurrence labels, so frozen vertices are preserved.
For color \(i_\ell\), the vertex \((i_\ell,2)=\ell_{\min}^+\) becomes
\((i_\ell,1)\) under \(\Phi_\ell\), which is precisely the frozen boundary
vertex of the color \(i_\ell\)-chain in \(Q_{\bfi}\setminus\{\ell\}\).
Thus \(\Phi_\ell\) sends frozen vertices of \(\widehat Q_\ell\) exactly to
frozen vertices of \(Q_{\bfi}\setminus\{\ell\}\).
\end{proof}
\begin{example}
\label{exam:mutation-subquivers}
We continue Example~\ref{exam_bfi}. First consider the two-color subquiver
$Q_{7,9}$. After applying the mutation sequence $\widetilde\mu_9$, we
obtain the quiver shown in Figure~\ref{fig:quiver79}.

\begin{figure}[htbp]
\centering
\begin{tikzpicture}[
    x=1.0cm,
    y=0.9cm,
    >={Stealth[length=1.7mm]},
    every node/.style={
        circle,
        draw,
        minimum size=5.5mm,
        inner sep=0pt,
        font=\small
    },
    every edge/.style={draw, ->, thick}
]

\node (2) at (0.4,1.2) {$2$};
\node (3) at (1.6,1.2) {$3$};
\node (7) at (4.0,1.2) {$7$};

\node (4) at (2.0,0) {$4$};
\node (6) at (3.2,0) {$6$};
\node (9) at (5.2,0) {$9$};

\draw (3) edge (2);
\draw (7) edge (3);
\draw (4) edge (6);
\draw (9) edge (6);
\draw (6) edge (7);

\end{tikzpicture}
\caption{The quiver obtained from $Q_{7,9}$ after applying $\widetilde\mu_9$.}
\label{fig:quiver79}
\end{figure}

Next take the word $\bfj=(12231321)$ and consider the subquiver
$Q_{7,8}$. After applying the mutation sequence $\widetilde\mu_8$, we get
the quiver on the left-hand side of Figure~\ref{fig:quiver78}. In the
construction of $\widehat Q_8$, the vertex $1$ is deleted, while the
vertices $2$ and $5$ become frozen. Thus, after deleting the vertex $1$
and the arrows incident with it, and also removing the arrows between
frozen vertices, we obtain the quiver on the right-hand side of
Figure~\ref{fig:quiver78}.

\begin{figure}[htbp]
\centering
\begin{tikzpicture}[
    x=0.95cm,
    y=0.9cm,
    >={Stealth[length=1.7mm]},
    every node/.style={
        circle,
        draw,
        minimum size=5.5mm,
        inner sep=0pt,
        font=\small
    },
    every edge/.style={draw, ->, thick},
    frozen/.style={
        circle,
        draw,
        double,
        minimum size=5.5mm,
        inner sep=0pt,
        font=\small
    }
]

\node (1L) at (-0.4,2.4) {$1$};
\node[frozen] (5L) at (2.8,2.4) {$5$};
\node (8L) at (4.8,2.4) {$8$};

\node[frozen] (2L) at (0.4,1.2) {$2$};
\node (3L) at (1.6,1.2) {$3$};
\node (7L) at (4.0,1.2) {$7$};

\draw (1L) edge (5L);
\draw (8L) edge (5L);
\draw (2L) edge (8L);
\draw (5L) edge (2L);
\draw (8L) edge (7L);
\draw (2L) edge (1L);
\draw (3L) edge (2L);
\draw (7L) edge (3L);

\node[draw=none] at (2.3,-0.55) {$\widetilde\mu_8(Q_{7,8})$};

\node[frozen] (5R) at (8.0,2.4) {$5$};
\node (8R) at (10.0,2.4) {$8$};

\node[frozen] (2R) at (5.6,1.2) {$2$};
\node (3R) at (6.8,1.2) {$3$};
\node (7R) at (9.2,1.2) {$7$};

\draw (8R) edge (5R);
\draw (2R) edge (8R);
\draw (8R) edge (7R);
\draw (3R) edge (2R);
\draw (7R) edge (3R);

\node[draw=none] at (7.8,-0.55) {$\widehat Q_8$};

\end{tikzpicture}
\caption{The mutation of $Q_{7,8}$ and the corresponding quiver $\widehat Q_8$.}
\label{fig:quiver78}
\end{figure}
\end{example}

\subsection{T-system}
Throughout this subsection, let $\bfi=(i_1,\dots,i_\ell)$ be an expression
of $b\in\operatorname{Br}^+$. We assume that the quantum cluster algebra
$\cA(\bt_\bfi)$ is contained in $\widehat{\cA}(b)$ and that all cluster
monomials of $\cA(\bt_\bfi)$ belong to the global basis of $\widehat{\cA}(b)$.  We will prove that, for every $i$-box
$[a,c]\subset [1,\ell]$, the quantum minor $D_{\bfi}[a,c]$ is a cluster
variable in $\cA_{\bfi}$ and these quantum minors satisfy $T$-system. 
\begin{lemma}
\label{lem_mutation_Daell}
 Let $\bfi=(i_1,\dots,i_\ell)$ be an expression
of $b\in\operatorname{Br}^+$. We assume that the quantum cluster algebra
$\cA(\bt_\bfi)$ is contained in $\widehat{\cA}(b)$ and that all cluster
monomials of $\cA(\bt_\bfi)$ belong to the global basis of $\widehat{\cA}(b)$.
After applying the mutation sequence $\widetilde{\mu}_{\ell}$, then the quantum minors
\[
D_{\bfi}[a,\ell^-]= \widetilde{\mu}_\ell(D_{\bfi,a^+})
\qquad (i_a=i_\ell,\ a< \ell)
\]
are cluster variables in the seed
\(
\widetilde{\mu}_{\ell}(\mathbf t_{\bfi}).
\)
\end{lemma}

\begin{proof}
Let
\[
a_0=\ell,\qquad a_r=\ell^{-r}\quad(r\geq 1)
\]
be the predecessor chain of the vertex \(\ell\) of color \(i_\ell\).
Thus the vertices mutated by \(\widetilde{\mu}_\ell\) are
\[
a_0,a_1,a_2,\ldots
\]
as long as they are mutable.

We prove by induction on \(r\geq 1\) that
\[
\mu_{a_{r-1}}\cdots\mu_{a_0}(D_{\bfi,a_{r-1}})
=
D_{\bfi}[a_r,\ell^-].
\]
The case \(r=1\) says
\[
\mu_\ell(D_{\bfi,\ell})
=
D_{\bfi}[\ell^-,\ell^-].
\]

Indeed, by Lemma~\ref{lem:embedding-Qell}, at the moment when we mutate at
\(\ell\), the only outgoing horizontal arrow from \(\ell\) is
\[
\ell\longrightarrow \ell^-,
\]
and all arrows ending at \(\ell\) are ordinary arrows. Hence the exchange
relation has the form
\[
D_{\bfi,\ell}\,\mu_\ell(D_{\bfi,\ell})
=
q^A D_{\bfi,\ell^-}+q^B M,
\]
where \(M\) is a cluster monomial involving only the variables attached to
the ordinary arrows ending at \(\ell\).

Since all cluster monomials of \(\cA_{\bfi}\) belong to the global basis,
\(\mu_\ell(D_{\bfi,\ell})\) is a global basis element. By
Lemma~\ref{lem_prodglobal}, we have
\[
\mathcal L_{\bfi}(D_{\bfi,\ell})
+
\mathcal L_{\bfi}(\mu_\ell(D_{\bfi,\ell}))
=
\max\{
\mathcal L_{\bfi}(D_{\bfi,\ell^-}),
\mathcal L_{\bfi}(M)
\}.
\]
The \(\ell\)-th coordinate of
\(\mathcal L_{\bfi}(D_{\bfi,\ell^-})\) is nonzero, whereas the
\(\ell\)-th coordinate of \(\mathcal L_{\bfi}(M)\) is zero. Hence the
maximum is
\[
\mathcal L_{\bfi}(D_{\bfi,\ell^-}).
\]
Therefore
\[
\mathcal L_{\bfi}(\mu_\ell(D_{\bfi,\ell}))
=
\mathcal L_{\bfi}(D_{\bfi,\ell^-})
-
\mathcal L_{\bfi}(D_{\bfi,\ell})
=
\mathcal L_{\bfi}(D_{\bfi}[\ell^-,\ell^-]).
\]
Since both sides are global basis elements with the same
\(\bfi\)-Lusztig parameter, we get
\[
\mu_\ell(D_{\bfi,\ell})
=
D_{\bfi}[\ell^-,\ell^-].
\]

Now assume that the claim has been proved up to \(r-1\). We prove it for
\(r\). At the moment when we mutate at \(a_{r-1}\), Lemma~\ref{lem:embedding-Qell}
shows that the outgoing arrows from \(a_{r-1}\) are horizontal arrows,
while the incoming arrows are ordinary arrows. Hence the exchange relation
for \(D_{\bfi,a_{r-1}}\) has the form
\[
D_{\bfi,a_{r-1}}\,
\mu_{a_{r-1}}\cdots\mu_{a_0}(D_{\bfi,a_{r-1}})
=
q^A D_{\bfi,a_r}\,D_{\bfi}[a_{r-1},\ell^-]
+
q^B M',
\]
where \(M'\) is the cluster monomial coming from ordinary arrows.

By the induction hypothesis,
\[
D_{\bfi}[a_{r-1},\ell^-]
\]
is the cluster variable obtained at the previous step. By the same
Lusztig-parameter comparison as above, the maximum in
Lemma~\ref{lem_prodglobal} is attained by the horizontal term
\[
D_{\bfi,a_r}\,D_{\bfi}[a_{r-1},\ell^-].
\]
Therefore
\[
\begin{aligned}
&\mathcal L_{\bfi}
\bigl(
\mu_{a_{r-1}}\cdots\mu_{a_0}(D_{\bfi,a_{r-1}})
\bigr) \\
&\qquad =
\mathcal L_{\bfi}(D_{\bfi,a_r})
+
\mathcal L_{\bfi}(D_{\bfi}[a_{r-1},\ell^-])
-
\mathcal L_{\bfi}(D_{\bfi,a_{r-1}}).
\end{aligned}
\]
By the defining additivity of Lusztig parameters for these quantum minors,
the right-hand side is precisely
\[
\mathcal L_{\bfi}(D_{\bfi}[a_r,\ell^-]).
\]
Again both elements are global basis elements, so they are equal:
\[
\mu_{a_{r-1}}\cdots\mu_{a_0}(D_{\bfi,a_{r-1}})
=
D_{\bfi}[a_r,\ell^-].
\]

This proves the desired identity for every predecessor \(a=a_r\) of
\(\ell\). Equivalently, for every \(a<\ell\) with \(i_a=i_\ell\), we have
\[
\widetilde{\mu}_\ell(D_{\bfi,a^+})
=
D_{\bfi}[a,\ell^-].
\]
Hence these quantum minors are cluster variables in the seed
\(\widetilde{\mu}_\ell(\mathbf t_{\bfi}).
\)
\end{proof}

\begin{theorem}
\label{thm:Mk-minors}
Let \(\bfi=(i_1,\dots,i_\ell)\) be an expression of
\(b\in\operatorname{Br}^+\). Assume that the quantum cluster algebra
\(\mathcal A(\mathbf t_\bfi)\) is contained in \(\widehat{\mathcal A}(b)\)
and that all cluster monomials of \(\mathcal A(\mathbf t_\bfi)\) belong
to the global basis of \(\widehat{\mathcal A}(b)\).

Then, for every \(k\in[1,\ell]\) and every vertex \(a\) of \(\widehat Q_k\),
one has
\begin{equation}
\label{eq_minorsMk}
M_k(D_{\bfi,a})
=
D_{\bfi}\bigl[a^{-k[i_a]},\,a_{\max}^{-k[i_a]}\bigr].
\end{equation}
Here \(a_{\max}=(i_a,n_{i_a})\) denotes the rightmost vertex of color
\(i_a\) in the word \(\bfi\). We use the convention that
\(D_{\bfi}[u,v]=1\) if the interval \([u,v]\) is empty.
\end{theorem}

\begin{proof}
For \(0\leq k\leq \ell\), write
\[
\bfi_k=(i_1,\dots,i_k).
\]
We first prove, by descending induction on \(k\), that there is an
isomorphism of quivers
\begin{equation}
\label{eq_isoPhik}
\Phi_k:\widehat Q_k\xrightarrow{\sim}Q_{\bfi_{k-1}}
\end{equation}
given in occurrence notation by
\[
\Phi_k(i,r)=(i,r-k[i]).
\]
Moreover, under this isomorphism, the frozen vertices of \(\widehat Q_k\)
correspond exactly to the frozen vertices of \(Q_{\bfi_{k-1}}\). That is, after deleting
the vertices \((i,p)\) with \(p\leq k[i]\), the remaining vertices
\[
(i,k[i]+1),\dots,(i,n_i)
\]
are relabelled as
\[
(i,1),\dots,(i,n_i-k[i])
\]
by the map
\[
\Phi_k(i,r)=(i,r-k[i]).
\]
Thus \(\Phi_k\) is a relabelling isomorphism.

For \(k=\ell\), this is precisely Lemma~\ref{lem:embedding-Qell}. Assume
the assertion is known for \(k+1\). Thus we have an isomorphism
\[
\Phi_{k+1}:\widehat Q_{k+1}\xrightarrow{\sim} Q_{\bfi_k}.
\]
The additional mutation sequence needed to pass from \(M_{k+1}\) to
\(M_k\) is the mutation sequence along the remaining vertices of color
\(i_k\), namely
\[
(i_k,k[i_k]+1)\longleftarrow\cdots\longleftarrow(i_k,n_{i_k}-1)\longleftarrow(i_k,n_{i_k})
\]
Applying \(\Phi_{k+1}\), this sequence becomes
\[
(i_k,2)\longleftarrow\cdots\longleftarrow(i_k,n_{i_k}-1-(k+1)[i_k])\longleftarrow(i_k,n_{i_k}-(k+1)[i_k])
\]
In the prefix word \(\bfi_k\), the terminal vertex \(k\) of color \(i_k\)
is \((i_k,n_{i_k}-(k+1)[i_k])\) in occurrence notation. Hence the displayed sequence $\widetilde\mu_k$ is exactly the mutation
sequence $\mu_{k_{\min}^+}\cdots\mu_{k}$ in
\(Q_{\bfi_k}\).

Therefore, under \(\Phi_{k+1}\), the mutation process producing
\(\widehat Q_k\) from \(\widehat Q_{k+1}\) corresponds to the mutation
process in Lemma~\ref{lem:embedding-Qell} applied to the word \(\bfi_k\).
That lemma gives an isomorphism
\[
\widehat Q_k\xrightarrow{\sim}Q_{\bfi_{k-1}},
\]
and, in occurrence notation, this is precisely
\[
\Phi_k(i,r)=(i,r-k[i]).
\]
The same lemma also shows that frozen vertices are identified with frozen
vertices. This proves \eqref{eq_isoPhik} for all \(k\).\\

\noindent
We now prove \eqref{eq_minorsMk}, again by descending induction on \(k\).

First consider \(k=\ell\). If \(i_a=i_\ell\), then the formula follows
from Lemma~\ref{lem_mutation_Daell}. Indeed, that
lemma gives
\[
M_\ell(D_{\bfi,a})
=
D_{\bfi}[a^-,\ell^-].
\]
Since \(\ell[i_\ell]=1\), this is exactly
\[
D_{\bfi}\bigl[a^{-\ell[i_a]},a_{\max}^{-\ell[i_a]}\bigr].
\]
If \(i_a\neq i_\ell\), then \(\ell[i_a]=0\), and the mutation sequence
\(\widetilde\mu_\ell\) does not mutate the vertex \(a\). Hence
\[
M_\ell(D_{\bfi,a})=D_{\bfi,a}
=
D_{\bfi}[a,a_{\max}]
=
D_{\bfi}\bigl[a^{-\ell[i_a]},a_{\max}^{-\ell[i_a]}\bigr].
\]
Thus the formula holds for \(k=\ell\).

Assume now that the formula has been proved for \(k+1\). We prove it for
\(k\). Let \(a\) be a vertex of \(\widehat Q_k\).
First suppose that \(i_a\neq i_k\). Then
\[
k[i_a]=(k+1)[i_a].
\]
Moreover, the additional mutations from \(M_{k+1}\) to \(M_k\) are
performed only at vertices of color \(i_k\). Hence the cluster variable
attached to \(a\) is not mutated during this extra sequence. Therefore,
by the induction hypothesis,
\[
M_k(D_{\bfi,a})
=
M_{k+1}(D_{\bfi,a})
=
D_{\bfi}\bigl[a^{-(k+1)[i_a]},a_{\max}^{-(k+1)[i_a]}\bigr].
\]
Since \(k[i_a]=(k+1)[i_a]\), this is exactly
\[
D_{\bfi}\bigl[a^{-k[i_a]},a_{\max}^{-k[i_a]}\bigr].
\]

It remains to consider the case \(i_a=i_k\). Write
\[
a=(i_k,p),
\qquad
p>k[i_k].
\]

For \(s=k,k+1\), set
\[
a^{[s]}:=\Phi_s(a).
\]
Since \(i_a=i_k\), we have
\[
k[i_a]=(k+1)[i_a]+1.
\]
Hence the two shifted vertices satisfy
\[
a^{[k]}=(a^{[k+1]})^-,
\qquad
(a^{[k]})^+=a^{[k+1]}.
\]
Moreover, under the map \(\Phi_{k+1}\), the rightmost vertex of color
\(i_k\) in the word $\bfi$ is identified with the terminal vertex
\(k\) of the prefix word \(\bfi_k\). After passing to \(\Phi_k\), this
vertex is shifted one step to the left, and hence is identified with
\(k^-\). Thus
\[
(a_{\max})^{[k+1]}=k,
\qquad
(a_{\max})^{[k]}=k^-.
\]
For an adjacent color \(i\neq i_k\), let \((i,n_i)\) be the rightmost
vertex of color \(i\) in the original word \(\bfi\). Since
\(k[i]=(k+1)[i]\), we have
\begin{equation}\label{eq_amax'}
\Phi_{k+1}(i,n_i)
=
(i,n_i-(k+1)[i])
=
k(i)^-.
\end{equation}
Therefore, let $c$ be a vertex in $\widehat{Q}_{k+1}$ with $i_c\neq i_k$, then the right endpoint $(c_{\max})^{[k+1]}$ of the cluster variable $D_\bfi[c^{[k+1]},(c_{\max})^{[k+1]}]$ is \(k(i_c)^-\).

Let us examine the local configuration at the vertex \(a=(i_k,p)\) in the
quiver
\[
\mu_{a^+}\cdots \mu_{k_{\max}}(\widehat Q_{k+1}).
\]
At this stage, the horizontal arrows around \(a\) have the form
\[
a^-\longleftarrow a\longrightarrow a^+.
\]
Under \(\Phi_{k+1}\), this is induced from the horizontal configuration
\[
(a^{[k+1]})^-
\longleftarrow
a^{[k+1]}
\longrightarrow
(a^{[k+1]})^+.
\]

By the isomorphism
\[
\widehat Q_{k+1}\simeq Q_{\bfi_k},
\]
the ordinary arrows with target \(a\) in \(\widehat Q_{k+1}\) correspond
to the ordinary arrows with target \(a^{[k+1]}\) in \(Q_{\bfi_k}\). Let
\(c\in[1,k]\) be such that \(c_{i_c i_a}=-1\), and consider the
two-color subquiver of \(Q_{\bfi_k}\) with colors \(i_c\) and
\(i_{a^{[k+1]}}\).

By Lemma~\ref{lem_first_ordinary}, the first ordinary arrow in this
two-color subquiver is oriented from a vertex of color \(i_c\) to a
vertex of color \(i_{a^{[k+1]}}\). We write this arrow as
\[
j^1\longrightarrow k^1,
\qquad
i_{j^1}=i_c,
\qquad
i_{k^1}=i_{a^{[k+1]}}.
\]

Now suppose that
\[
j^t\longrightarrow a^{[k+1]}
\]
is an ordinary arrow in the quiver
\[
\mu_{(a^{[k+1]})^+}\cdots \mu_k(Q_{\bfi_k}).
\]
Then, by the local configuration in
Figure~\ref{fig:four-local-configurations}\,(2), this is equivalent to
\[
(j^t)^- < (a^{[k+1]})^- < j^t.
\]
Hence \(j^t\) is the first vertex of color \(i_p\) lying to the right of
\((a^{[k+1]})^-\). Equivalently,
\begin{equation}\label{eq_arrowjt}
j^t
=
(a^{[k+1]})^-(i_p)^+
=
a^{[k]}(i_p)^+.
\end{equation}

For the vertices of color \(i_k\), we prove the formula \eqref{eq_minorsMk} by induction along
the additional mutation sequence
\[
(i_k,k[i_k]+1)\longleftarrow\cdots\longleftarrow(i_k,n_{i_k}).
\]
Assume that the formula has already been proved for the vertices mutated
after \(a\), in particular for \(a^+\).
For this vertex, \(a^+\) does not exist, and the corresponding factor in
the exchange relation is interpreted as \(1\). Then the cluster variable attached
to \(a^+\) is
\[
\mu_{a^+}
\left(
D_{\bfi}[(a^+)^{[k+1]},k]
\right)
=
D_{\bfi}[(a^+)^{[k]},k^-].
\]
Since
\[
(a^+)^{[k+1]}=(a^{[k+1]})^+,
\qquad
(a^+)^{[k]}=(a^{[k]})^+=a^{[k+1]},
\]
this becomes
\[
\mu_{a^+}
\left(
D_{\bfi}[(a^{[k+1]})^+,k]
\right)
=
D_{\bfi}[a^{[k+1]},k^-].
\]
By Lemma~\ref{lem:embedding-Qell} and the local description of the
two-color subquivers, at the moment of mutation at \(a\), the outgoing
arrows from \(a\) are precisely the two horizontal arrows, while the
incoming arrows are the ordinary arrows indexed by adjacent colors
\(i\) with \(c_{ii_k}=-1\). Therefore the exchange relation has the form
\begin{equation}
\label{eq_preTsystem}
\begin{aligned}
&
D_{\bfi}[a^{[k+1]},k]\,
\mu_a\left(D_{\bfi}[a^{[k+1]},k]\right)
\\
&\quad =
q^A
D_{\bfi}[a^{[k+1]},k^-]\,
D_{\bfi}[(a^-)^{[k+1]},k]
\\
&\qquad
+
q^B
\prod_{\substack{i\in I\\ c_{ii_k}=-1}}
D_{\bfi}[a^{[k]}(i)^+,k(i)^-].
\end{aligned}
\end{equation}
The first two factors on the right-hand side correspond
to the horizontal arrows, while the product corresponds to the ordinary
arrows with target \(a\).
The last term follows from \eqref{eq_amax'} and \eqref{eq_arrowjt}.
Here the product is taken over those adjacent colors \(i\) for which the
vertex \(a^{[k]}(i)^+\) exists; if no such vertex exists, the corresponding
factor is omitted. 

We compare the \(\bfi\)-Lusztig parameters of the two terms on the
right-hand side of \eqref{eq_preTsystem}. All factors in both exchange monomials have support contained in
\([1,k]\). Moreover, at the largest relevant coordinate \(k\), the first
exchange monomial has coordinate \(1\), whereas the ordinary-arrow
monomial has coordinate \(0\). Hence, with respect to the order of $\bfi$-Lusztig parameters, the first exchange monomial gives the
maximum.

The mutated variable is a cluster variable, hence a global basis element by
the assumption on \(\mathcal A(\mathbf t_\bfi)\). Applying Lemma~\ref{lem_prodglobal} to the exchange relation
\eqref{eq_preTsystem}, we obtain 

\[
\begin{aligned}
&\mathcal L_{\bfi}
\left(
\mu_a\left(D_{\bfi}[a^{[k+1]},k]\right)
\right)\\
=&
\mathcal L_{\bfi}(D_{\bfi}[a^{[k+1]},k^-])
+
\mathcal L_{\bfi}(D_{\bfi}[(a^-)^{[k+1]},k])
-
\mathcal L_{\bfi}(D_{\bfi}[a^{[k+1]},k])
\\
=&
\mathcal L_{\bfi}(D_{\bfi}[a^{[k]},k^-]).
\end{aligned}
\]
Here note that \((a^-)^{[k+1]}=a^{[k]}\). Since global
basis elements are uniquely determined by their \(\bfi\)-Lusztig
parameters, we obtain
\[
\mu_a
\left(D_{\bfi}[a^{[k+1]},k]\right)
=
D_{\bfi}[a^{[k]},k^-].
\]
This is exactly
\[
\mu_a
\left(
D_{\bfi}
[
a^{-(k+1)[i_a]},
a_{\max}^{-(k+1)[i_a]}
]
\right)
=
D_{\bfi}
[
a^{-k[i_a]},
a_{\max}^{-k[i_a]}
],
\]
because
\[
a^{[k+1]}=a^{-(k+1)[i_a]},
\qquad
a^{[k]}=a^{-k[i_a]},
\qquad
k=a_{\max}^{-(k+1)[i_a]},
\qquad
k^-=a_{\max}^{-k[i_a]}.
\]
This proves the desired formula for the vertex \(a\). 

Together with the case \(i_a\neq i_k\), this proves the induction step
from \(k+1\) to \(k\). Hence the formula holds for every \(k\in[1,\ell]\)
and every vertex \(a\) of \(\widehat Q_k\).
\end{proof}

\begin{theorem}
\label{thm_Tsystem}
 let $\bfi=(i_1,\dots,i_\ell)$ be an expression
of $b\in\operatorname{Br}^+$. We assume that the quantum cluster algebra
$\cA(\bt_\bfi)$ is contained in $\widehat{\cA}(b)$ and that all cluster
monomials of $\cA(\bt_\bfi)$ belong to the global basis of $\widehat{\cA}(b)$. Then the quantum minors satisfy the
following quantum $T$-system. For every $i$-box $[a,c]$, one has
\begin{equation}
\label{eq_Tsystem1}
D_{\bfi}[a^+,c]\,
D_{\bfi}[a,c^-]
=
q^A
D_{\bfi}[a,c]\,
D_{\bfi}[a^+,c^-]
+
q^B
\prod_{\substack{j\in I\\ c_{i_a j}c_{j i_a}=1}}
D_{\bfi}[a(j)^+,c(j)^-],
\end{equation}
for some $A,B\in\mathbb Z$. Here we use the convention that
$D_{\bfi}[u,v]=1$ whenever $u>v$.
\end{theorem}

\begin{proof}
The relation is obtained from the exchange relation appearing in
\eqref{eq_preTsystem}. Indeed, by Theorem~\ref{thm:Mk-minors}, the
cluster variables produced along the mutation sequence $M_k$ are
identified with the quantum minors
\[
D_{\bfi}[u,v].
\]
Under this identification, the two monomials in the exchange relation
correspond respectively to
\[
D_{\bfi}[a,c]\,
D_{\bfi}[a^+,c^-]
\]
and
\[
\prod_{\substack{j\in I\\ c_{i_a j}c_{j i_a}=1}}
D_{\bfi}[a(j)^+,c(j)^-].
\]
Here the mutated variable is $D_{\bfi}[a^+,b]$, and the new variable obtained after mutation is $D_{\bfi}[a,b^-]$. Therefore the exchange relation becomes
precisely \eqref{eq_Tsystem1}, up to the powers $q^A$ and $q^B$ coming
from the quantum commutation factors.
\end{proof}

\section{Monoidal categorification in the simply-laced Dynkin type}\label{sec_monoidal}
In this section, we will focus on the simply-laced case and assume that the Cartan matrix $C$ is ADE types.

\subsection{Q-Datum}
Let \( \Delta \) be a simply-laced Dynkin diagram with vertex set \( I \) and edge set \( \Delta_1 \). A \emph{height function} \( \xi = (\xi_i)_{i \in I} \) on \( \Delta \) is defined by the condition:

\[
|\xi_i - \xi_j| = 1 \quad \text{for any } i, j \in I \text{ such that } c_{ij} = -1.
\]

The height \( \xi \) induces an orientation on \( \Delta \), where an arrow \( i \to j \) exists if and only if \( \xi_i > \xi_j \) and \(c_{ij}=-1\).

\begin{definition}
A \emph{$Q$-datum} is a pair \( (\Delta, \xi) \), where \( \Delta \) is a simply-laced Dynkin diagram and \( \xi \) is a height function on \( \Delta \).
\end{definition}

For a vertex \( i \in I \) that is a \emph{source} in the quiver associated with \( \xi \) (i.e., there are no arrows \( j \to i \)), the \emph{reflection} \( s_i \) acts on the height function by:
\[
s_i(\xi_j) = \xi_j - 2 \delta_{ij}.
\]

Let \( w_0 \) be the longest element of the Weyl group associated with \( \Delta \), with length \( \ell \), and let \( \underline{w_0} = (i_1, \ldots, i_\ell) \) be a reduced expression for \( w_0 \). The expression \( \underline{w_0} \) is \emph{adapted} to the \( Q \)-datum \( (\Delta, \xi) \) if, for each \( k \leq \ell \), the vertex \( i_k \) is a source in the quiver associated with the height:
\[
s_{i_{k-1}} \cdots s_{i_1} \xi.
\]

For an adapted reduced expression \( \underline{w_0} \), we extend it to an infinite sequence \( \widehat w_0 = \{ i_k \}_{k \in \mathbb{Z}} \) by:
\[
i_{k + \ell} = i_k^*, \quad \text{where } i^* \text{ satisfies } w_0(\alpha_i) = -\alpha_{i^*}.
\]

It follows that for any \( k \in \mathbb{Z} \):
\[
s_{i_k} \cdots s_{i_{k + \ell}} = w_0.
\]

Additionally, the action of \( w_0 \) on the height function satisfies:
\[
(w_0 \xi)_i = \xi_{i^*} - h,
\]

where \( h \) is the Coxeter number of \( \Delta \).

Define the set:
\begin{equation}\label{eq_DeltaHat}
\widehat{\Delta} := \left\{ (i, p) \in I \times \mathbb{Z} \mid p - \xi_i \in 2 \mathbb{Z} \right\}.
\end{equation}

For each integer \( k \), define the subset:
\begin{equation}\label{eq_DeltaQk}
\widehat{\Delta}_Q[k] := \left\{ (i, p) \in \widehat{\Delta} \mid \xi_{i^*} - (k + 1)h < p \leq \xi_i - k h \right\}.
\end{equation}

Let \( \widehat{\Delta}_Q := \widehat{\Delta}_Q[0] \), and define the negative part:
\[
\widehat{\Delta}_{\leq \xi}:= \left\{ (i, p) \in \widehat{\Delta} \mid p \leq \xi_i \right\}.
\]

For an adapted infinite sequence \( \widehat w_0 = \{ i_k \}_{k \in \mathbb{Z}} \), define the sequence of integers \( p_k \) by:

\begin{equation}\label{eq_pk}
p_k = \begin{cases}
    \left( s_{i_{k-1}} \cdots s_{i_1} \xi \right)_{i_k} & \text{if } k > 0, \\
    \left( s_{i_k}^{-1} \cdots s_{i_0}^{-1} \xi \right)_{i_k} & \text{if } k \leq 0.
\end{cases}
\end{equation}

\begin{theorem}[\cite{kashiwara2024monoidal}, Proposition~6.11]
\label{theo_w0Delta}
Let
\[
\widehat w_0=(i_k)_{k\in\mathbb Z}
\]
be an adapted sequence associated with a \(Q\)-datum \((\Delta,\xi)\). Then
the assignment
\[
\varphi_{\widehat w_0}:\mathbb Z\longrightarrow \widehat{\Delta},
\qquad
k\longmapsto (i_k,p_k),
\]
is a bijection. Equivalently, the adapted sequence \(\widehat w_0\), together
with the integers \(p_k\), gives an enumeration of the vertex set
\(\widehat{\Delta}\).
\end{theorem}

There exists a bijection:
\begin{equation}\label{eq_phi}
\phi: \widehat{\Delta} \to R^+ \times \mathbb{Z},
\end{equation}

such that \( \phi((i, \xi_i)) = (\gamma_i, 0) \), where \( \gamma_i \) is the root corresponding to the dimension vector of the injective representation \( I_i \) of the quiver associated with the \( Q \)-datum \( (\Delta, \xi) \). Furthermore, if \( \phi((i, p)) = (\alpha, k) \), then:
\[
\phi((i, p + 2)) = \begin{cases}
    (\tau \alpha, k) & \text{if } \tau \alpha \in R^+, \\
    (-\tau \alpha, k + 1) & \text{if } \tau \alpha \in R^-,
\end{cases}
\]
where \( \tau \) is the Coxeter element for the quiver.

\subsection{Quantum loop algebras}
This section introduces the quantum loop algebra \( U_q(L\mathfrak{g}) \) associated with a simple Lie algebra \( \mathfrak{g} \), along with its module categories and related algebraic structures.

\subsubsection{Modules of the quantum loop algebra}
Let \( \mathscr{C} \) denote the category of finite-dimensional \( U_q(L\mathfrak{g}) \)-modules of type 1. Each simple module \( M \in \mathscr{C} \) is characterized by a tuple of Drinfeld polynomials \( (P_i)_{i \in I} \), where \( P_i \in 1 + q \mathbb{Z}[q] \). The roots of these polynomials determine a dominant monomial:
\[
m = \prod_{(i, a)} Y_{i, a},
\]
where the product is over the roots \( a \) of \( P_i \). If \( a = q^k \), we write \( Y_{i, k} \) for \( Y_{i, q^k} \).

Define \( \mathcal{M} \) as the set of all dominant monomials, and let \( \mathcal{M}^+ \subset \mathcal{M} \) consist of monomials of the form:
\[
m = \prod_{(i, k) \in \widehat{\Delta}} Y_{i, k}^{u_{i, k}}.
\]

For each pair \( (i, p) \in I\times\ZZ \) such that \( (i, p - d_i) \in \widehat{\Delta} \), define the element:
\begin{equation}\label{eq_Aip}
A_{i, p} = Y_{i, p - d_i} Y_{i, p + d_i} \prod_{\substack{(j, s) \in \widehat{\Delta} \ j \sim i, |s - p| < d_i}} Y_{j, s}^{-1}.
\end{equation}

The \emph{Hernandez--Leclerc category} \( \mathscr{C}^{\mathbb{Z}} \) is the Serre subcategory of \( \mathscr{C} \) generated by simple modules \( L(m) \) for \( m \in \mathcal{M}^+ \). For a height\( \xi'\), the subcategory \( \mathscr{C}_{\leq \xi} \subset \mathscr{C}^{\mathbb{Z}} \) is the Serre subcategory generated by simple modules \( L(m) \) with:
\[
m = \prod_{(i, p) \in \widehat{\Delta}_{\leq \xi'}} Y_{(i, p)}^{u_{(i, p)}}.
\]

The simple module \( L(Y_{i, k}) \) is called a \emph{fundamental module}. For an interval \( [a, b] \) such that \( (i, a), (i, b) \in \widehat{\Delta} \), the \emph{Kirillov-Reshetikhin module} is defined as:
\[
M^{(i)}[a, b] := L\left( Y_{i, a} Y_{i, a + 2} \cdots Y_{i, b} \right).
\]

\subsubsection{Quantum Grothendieck rings}
Assume the Cartan matrix \( C \) is simply-laced. The \emph{quantum Cartan matrix} \( C(q) \) is defined as:

\[
C(q)_{ij} = \begin{cases}
    q_i + q_i^{-1} & \text{if } i = j, \\
    [c_{ij}]_i & \text{if } i \neq j.
\end{cases}
\]
Since \( C \) is invertible, \( C(q) \) has an inverse matrix \( \widetilde{C}(q) = \sum_{u \in \mathbb{Z}} \widetilde{c}(u) z^u \). For any \( (i, p), (j, q) \in \widehat{\Delta} \), define:
\begin{equation}\label{eq_cN}
\mathcal{N}((i, p), (j, q)) := \widetilde{c}(p - q - d_i) - \widetilde{c}(p - q + d_i) + \widetilde{c}(q - p - d_i) - \widetilde{c}(q - p + d_i).
\end{equation}

The \emph{quantum torus} \( \mathcal{Y}_q \) is the \( \mathbb{Z}[q^{\pm 1/2}] \)-algebra generated by \( Y_{i, p}^{\pm 1} \) for \( (i, p) \in \widehat{\Delta} \), subject to the relations:
\begin{align}
    Y_{i, p} * Y_{j, q} &= q^{\mathcal{N}((i, p), (j, q))} Y_{j, q} Y_{i, p}, \label{eq:qtorus1} \\
    Y_{i, p} * Y_{i, p}^{-1} &= Y_{i, p}^{-1} * Y_{i, p} = 1. \label{eq:qtorus2}
\end{align}

For a vector \( \mathbf{a} = (a_{i, p}) \in \mathbb{Z}^{\widehat{\Delta}} \), define the monomial:
\[
m(\mathbf{a}) = q^{A_{\mathbf{a}}} \overrightarrow{\prod}_{(i, k) \in \widehat{\Delta}} Y_{(i, k)}^{a_{i, k}},
\]

where \( A_{\mathbf{a}} \) is chosen such that the product is independent of the ordering.\\

\noindent
For each \( i \in I \), the subalgebra \( \mathcal{K}_{i, q} \subset \mathcal{Y}_q \) is generated over \( \mathbb{Z}[q^{1/2}] \) by:

\[
\left\{ Y_{i, k} (1 + q^{-1} A_{i, k + d_i}^{-1}) \mid (i, k) \in \widehat{\Delta} \right\} \cup \left\{ Y_{j, l} \mid (j, l) \in \widehat{\Delta}, j \neq i \right\}.
\]

The \emph{quantum Grothendieck group} \( K_q(\mathscr{C}^{\mathbb{Z}}) \) is defined as:

\[
K_t(\mathscr{C}^{\mathbb{Z}}) := \bigcap_{i \in I} \mathcal{K}_{i, q}.
\]

\begin{theorem}[\cite{hernandez2015quantum}, \cite{kashiwara2024global}]\label{theo_quantumgroth}
There exists a \( \mathbb{Z}[q^{\pm 1/2}] \)-algebra isomorphism:
\[
\Phi: \widehat{\mathcal A} \to K_t(\mathscr{C}^{\mathbb{Z}}),
\]
mapping the generator \( x_{i, m} \) to the \( (q, t) \)-character of the simple module \( L(Y_{i^{*m}, p_i + m h}) \), where \( (i, p_i) = \phi^{-1}(\alpha_i, 0) \). The basis formed by the \( (q, t) \)-characters \( \chi_t(L(m)) \) of simple modules \( L(m) \) coincides with the global basis \( \widehat{\mathbf{B}}(\infty) \). See \cite{hernandez2015quantum} for the definition of \( (q, t) \)-characters. We denote this basis in $K_t(\mathscr C^\ZZ)$ by $\mathbf{L}$.
\end{theorem}

\subsection{Cluster structures on Bosonic extension algebras}

The goal of this section is to prove our main theorem: Conjecture~\ref{con_boscluster}
holds for any expression of any braid group element
$b\in \operatorname{Br}^+$.

We first recall the quantum cluster structure on the quantum Grothendieck
ring $K_t(\mathscr C_{\le \xi})$, which will serve as the starting point
for the proof.
\begin{proposition}{\cite[Theorem~6.6]{fujita2023isomorphisms}}\label{pro_quantumclusteriso}
    Let $\widehat w_0$ be an adapted sequences of a $Q$-datum $(\Delta,\xi)$ and $\xi'$ be another height of $\Delta$. Let $\bfi_{\xi'}$ be the subsequence $\varphi^{-1}(\widehat{\Delta}_{\leq\xi'})$ of $\widehat w_0$. Then $K_t(\mathscr C_{\leq \xi})$ is a quantum cluster algebra with initial seed
    \[(\{M^{(i)}[a,\xi'_i]\}_{(i,a)\in \Delta_{\leq \xi'}}, B_{\bfi_{\xi'}},\Lambda_{\bfi_{\xi'}}).\]
    Moreover, the cluster monomials are contained in the canonical basis $\mathbf{L}$ of $K_t(\mathscr C_{\leq \xi'})$. 
\end{proposition}

\begin{definition}\label{def_root_modules}
Let $(\Delta,\xi)$ be a $Q$-datum and let
$\widehat w_0=(i_k)_{k\in\ZZ}$ be an adapted expression. For each
$k\in\ZZ$, define the root module
\[
C_k^{\widehat w_0}:=L(Y_{i_k,p_k}),
\]
where $p_k$ is given in \eqref{eq_pk}.
\end{definition}

By \cite{kashiwara2024monoidal}, the inverse image of
$\chi_t(C_k^{\widehat w_0})$ under the isomorphism $\Phi$ is the
global basis element $E_k^{\widehat w_0}$ associated with the expression
$\widehat w_0$ of $\Delta_{\rm Gar}^\infty$.

For an interval $[a,c]\subset\ZZ$, let $\mathscr C^{[a,c]}$ be the full
subcategory of $\mathscr C^\ZZ$ generated by the modules
$C_k^{\widehat w_0}$ for $k\in[a,c]$, and closed under subquotients,
extensions, and tensor products. We denote by $\widehat w_0[a,c]$ the
subword of $\widehat w_0$ indexed by $[a,c]$ and by $[a,c]^{\operatorname{ex}}$ the subset of $[a,c]$ consisting of elements $u$ with $u^-<a$. 

\begin{theorem}
\label{thm_KtCab}
Let $\widehat w_0$ be an adapted expression associated with a
$Q$-datum $(\Delta,\xi)$. Then $K_t(\mathscr C^{[a,c]})$ admits a
quantum cluster algebra structure with initial seed
\[
\bt_{[a,c]}
=
\bigl(
\{\Phi(D_{\widehat w_0,s})\}_{s\in[a,c]},
B_{\widehat w_0[a,c]},
\Lambda_{\widehat w_0[a,c]},
[a,c]^{\operatorname{ex}}
\bigr).
\]
Moreover, all cluster monomials belong to the canonical basis $\mathbf{L}$ of
$K_t(\mathscr C^{[a,c]})$.
\end{theorem}

\begin{proof}
Write
\[
\widehat w_0[a,c]=(i_a,\dots,i_c).
\]
We first consider the semi-infinite interval \((-\infty,c]\). For each
\(j\in I\), let
\[
k_j=\max\{k\leq c\mid i_k=j\},
\]
and set
\[
\xi'(j)=p_{k_j}.
\]
Since \(\widehat w_0\) is adapted, \(\xi'\) is again a height function.
By the definition of the root modules, we have
\[
K_t(\mathscr C^{(-\infty,c]})
=
K_t(\mathscr C_{\leq \xi'}).
\]
Moreover, for every \(k\leq c\),
\[
M^{(i_k)}[p_k,\xi'_{i_k}]
=
\Phi(D_{\widehat w_0,k}).
\]
Therefore Proposition~\ref{pro_quantumclusteriso} gives a quantum cluster
algebra structure on \(K_t(\mathscr C^{(-\infty,c]})\), with initial seed
\(\bt_{(-\infty,c]}\), and all its cluster monomials belong to the
canonical basis \(\mathbf L\).

Now consider the finite interval \([a,c]\). By the explicit form of the
exchange matrix for the adapted word, we have
\[
B_{(-\infty,a-1]\times [a,c]^{\operatorname{ex}}}=0.
\]
Hence Proposition~\ref{pro:subcluster} implies that the seed
\(\bt_{[a,c]}\) defines a subcluster algebra
\[
\mathcal A(\bt_{[a,c]})
\subset
K_t(\mathscr C^{(-\infty,c]}).
\]
Moreover, every cluster monomial of \(\mathcal A(\bt_{[a,c]})\) is a
cluster monomial of \(K_t(\mathscr C^{(-\infty,c]})\), and hence belongs
to the canonical basis \(\mathbf L\).

We next prove that
\[
\mathcal A(\bt_{[a,c]})
\subset
K_t(\mathscr C^{[a,c]}).
\]
We argue by induction on the length of a mutation sequence. The initial
cluster variables
\[
\Phi(D_{\widehat w_0,s}),\qquad s\in[a,c],
\]
belong to \(K_t(\mathscr C^{[a,c]})\) by definition.

Assume that all cluster variables in a seed \(\bt\) obtained from
\(\bt_{[a,c]}\) belong to \(K_t(\mathscr C^{[a,c]})\). Let \(X_k\) be a
mutable variable of \(\bt\), and write the exchange relation as
\[
X_k\mu_k(X_k)=q^A X+q^B Y,
\]
where \(X\) and \(Y\) are cluster monomials in \(\bt\). By the induction
hypothesis, \(X_k,X,Y\in K_t(\mathscr C^{[a,c]})\). Since cluster
monomials in the semi-infinite cluster algebra belong to the canonical
basis, Lemma~\ref{lem_prodglobal} gives
\[
\mathcal L_{\widehat w_0}(\Phi^{-1}(X_k))
+
\mathcal L_{\widehat w_0}(\Phi^{-1}(\mu_k(X_k)))
=
\max\left\{
\mathcal L_{\widehat w_0}(\Phi^{-1}(X)),
\mathcal L_{\widehat w_0}(\Phi^{-1}(Y))
\right\}.
\]
The Lusztig parameters of \(X_k,X,Y\) are supported on \([a,c]\). Hence
the Lusztig parameter of \(\mu_k(X_k)\) is also supported on \([a,c]\).
By Proposition~\ref{prop_pbw_global} and the definition of
\(\mathscr C^{[a,c]}\), this implies
\[
\mu_k(X_k)\in K_t(\mathscr C^{[a,c]}).
\]
Thus
\[
\mathcal A(\bt_{[a,c]})
\subset
K_t(\mathscr C^{[a,c]}).
\]

Conversely, by Theorem~\ref{thm:Mk-minors}, applied to the seed
\(\bt_{[a,c]}\), the mutation sequence \(M_1\) produces the root modules
\[
C_k^{\widehat w_0}=\Phi(D_{\widehat{w}_0}[k,k]),\qquad k\in[a,c],
\]
as cluster variables. Since these root modules generate
\(K_t(\mathscr C^{[a,c]})\), we obtain
\[
K_t(\mathscr C^{[a,c]})
\subset
\mathcal A(\bt_{[a,c]}).
\]
Therefore
\[
K_t(\mathscr C^{[a,c]})
=
\mathcal A(\bt_{[a,c]}).
\]

Finally, as observed above, every cluster monomial of
\(\mathcal A(\bt_{[a,c]})\) is a cluster monomial of the semi-infinite
cluster algebra \(K_t(\mathscr C^{(-\infty,c]})\). Hence it belongs to the
canonical basis \(\mathbf L\). This proves the theorem.
\end{proof}

\begin{corollary}
\label{theo_K1kl}
Let \(\widehat w_0=(i_1,i_2,\ldots)\) be an adapted sequence associated with
a \(Q\)-datum. For every \(k\geq 1\), set
\[
\Delta_{\rm Gar}^k
=
\sigma_{i_1}\cdots \sigma_{i_{k\ell(w_0)}}.
\]
Then there is an isomorphism
\[
\widehat{\mathcal A}(\Delta_{\rm Gar}^k)
=
\widehat{\mathcal A}[0,k-1]
\simeq
K_t(\mathscr C^{[1,k\ell(w_0)]}).
\]
Moreover, under this isomorphism, the global basis of
\(\widehat{\mathcal A}(\Delta_{\rm Gar}^k)\) coincides with the canonical
basis \(\mathbf L\) of \(K_t(\mathscr C^{[1,k\ell(w_0)]})\).

In particular, \(\widehat{\mathcal A}(\Delta_{\rm Gar}^k)\) is a quantum
cluster algebra with initial seed associated with the expression
\[
(i_1,\ldots,i_{k\ell(w_0)}),
\]
and its cluster monomials belong to the global basis. Hence
Conjecture~\ref{con_boscluster} holds for the adapted expression of
\(\Delta_{\rm Gar}^k\) induced by \(\widehat w_0\).
\end{corollary}

\begin{proof}
Write \(\widehat w_0=(i_1,i_2,\ldots)\). For \(k\geq 1\), the word
\[
(i_1,\ldots,i_{k\ell(w_0)})
\]
is the adapted expression of \(\Delta_{\rm Gar}^k\). By the construction of
the bosonic extension algebra, we have
\[
\widehat{\mathcal A}[0,k-1]
=
\widehat{\mathcal A}(\Delta_{\rm Gar}^k).
\]
By Theorem~\ref{theo_quantumgroth}, the isomorphism $\Phi$ induces an isomorphism
\[
\widehat{\mathcal A}[0,k-1]
\simeq
K_t(\mathscr C^{[1,k\ell(w_0)]}),
\]
which identifies the global basis of
\(\widehat{\mathcal A}[0,k-1]\) with the canonical basis \(\mathbf L\) of
\(K_t(\mathscr C^{[1,k\ell(w_0)]})\).

By Theorem \ref{thm_KtCab}, the Grothendieck ring
\(K_t(\mathscr C^{[1,k\ell(w_0)]})\) is a quantum cluster algebra whose
cluster monomials belong to \(\mathbf L\). Transporting this cluster
structure through the above isomorphism gives a quantum cluster algebra
structure on \(\widehat{\mathcal A}(\Delta_{\rm Gar}^k)\), with initial
seed associated with the adapted expression
\[
\Delta_{\rm Gar}^k
=
\sigma_{i_1}\cdots\sigma_{i_{k\ell(w_0)}}.
\]
Since the isomorphism identifies \(\mathbf L\) with the global basis, the
cluster monomials belong to the global basis. Hence
Conjecture~\ref{con_boscluster} holds for the expression of
\(\Delta_{\rm Gar}^k\) induced by \(\widehat w_0\).
\end{proof}

\begin{theorem}
\label{theo_Dynkincase}
In simply-laced Dynkin type, Conjecture~\ref{con_boscluster} holds for every
$b\in\Br^+$.
\end{theorem}

\begin{proof}
Let $\widehat w_0=(i_1,i_2,\dots)$ be an adapted sequence associated with
a $Q$-datum $\mathscr L=(\Delta,\xi)$. For $m\ge1$, set
\[
\Delta_{\mathrm{Gar}}^m
=
\sigma_{i_1}\cdots \sigma_{i_{m\ell(w_0)}}.
\]
By Corollary~\ref{theo_K1kl}, the algebra
\[
\widehat{\mathcal A}(\Delta_{\mathrm{Gar}}^m)
\simeq
K_t(\mathscr C^{[1,m\ell(w_0)]})
\]
admits a quantum cluster algebra structure whose cluster monomials belong
to the global basis.

Since $\operatorname{Br}^+$ is a Garside monoid in finite type, by \cite[Corollary~6.9]{oh2025pbw} every
$b\in\operatorname{Br}^+$ is a right divisor of a sufficiently large
power of the Garside element $\Delta_{\mathrm{Gar}}$. Hence there exist
$m\ge1$ and $u\in\operatorname{Br}^+$ such that
\[
ub=\Delta_{\mathrm{Gar}}^m.
\]
Choose expressions $\bfk=(k_1,\dots,k_n)$ and
$\bfj=(j_1,\dots,j_\ell)$ of $u$ and $b$, respectively, and set
\[
\bfi=(\bfk,\bfj)
=
(k_1,\dots,k_n,j_1,\dots,j_\ell).
\]
Thus $\bfi$ is an expression of $\Delta_{\mathrm{Gar}}^m$. By
Corollary~\ref{theo_K1kl} and Theorem~\ref{theo_words}, the algebra
$\widehat{\mathcal A}(\Delta_{\mathrm{Gar}}^m)$ has a quantum cluster
structure with initial seed $\mathbf t_{\bfi}$, and all cluster monomials
belong to the global basis.

For $s>n$, the root vector attached to the $s$-th position of $\bfi$ is
obtained from the corresponding root vector of $\bfj$ by the braid action
$T_u$:
\[
E_s^{\bfi}=T_u(E_{s-n}^{\bfj}).
\]
Therefore the subalgebra of $\widehat{\mathcal A}(\Delta_{\mathrm{Gar}}^m)$
generated by the root vectors $E_s^{\bfi}$ with $s>n$ is precisely
$T_u\widehat{\mathcal A}(b)$.

We define the tail seed
\[
\mathbf t_{\bfj}^{\,u}
:=
\bigl(
(D_{\bfi}[a,n+\ell\})_{n<a\le n+\ell},
\Lambda_{\bfj},
B_{\bfj},
K^{\operatorname{ex}}
\bigr),
\]
where
\[
K^{\operatorname{ex}}
=
\{\,a\in\{n+1,\dots,n+\ell\}\mid a^->n\,\}.
\]
Here $\Lambda_{\bfj}$ and $B_{\bfj}$ are identified with the restrictions
of $\Lambda_{\bfi}$ and $B_{\bfi}$ to the tail positions
$\{n+1,\dots,n+\ell\}$ via the shift $a\mapsto a-n$. Thus
$\mathbf t_{\bfj}^{\,u}$ is naturally identified with the tail subseed of
$\mathbf t_{\bfi}$.

We claim that
\[
B_{[1,n]\times K^{\operatorname{ex}}}=0.
\]
Indeed, suppose that there were an arrow between a vertex
$k\in[1,n]$ and a vertex $j\in K^{\operatorname{ex}}$. Since
$j\in K^{\operatorname{ex}}$, we have $j^->n$. An ordinary arrow would
force one of the interlacing inequalities
\[
k^-<j^-<k<j
\qquad\text{or}\qquad
j^-<k^-<j<k,
\]
which is impossible because $k\le n<j^-$. A horizontal arrow is also
impossible: such an arrow across the cut would require the predecessor of
a tail vertex to lie in $[1,n]$, contradicting $j^->n$. Hence the claimed
vanishing holds.

By Proposition~\ref{pro:subcluster}, this vanishing implies
\[
\mathcal A(\mathbf t_{\bfj}^{\,u})
\subset
\mathcal A(\mathbf t_{\bfi})
=
\widehat{\mathcal A}(\Delta_{\mathrm{Gar}}^m).
\]
Moreover, the cluster monomials of $\mathcal A(\mathbf t_{\bfj}^{\,u})$
are cluster monomials of $\mathcal A(\mathbf t_{\bfi})$, and therefore
belong to the global basis of
$\widehat{\mathcal A}(\Delta_{\mathrm{Gar}}^m)$.

We next prove that
\[
\mathcal A(\mathbf t_{\bfj}^{\,u})
\subset
T_u\widehat{\mathcal A}(b).
\]
The initial cluster variables of $\mathbf t_{\bfj}^{\,u}$ belong to
$T_u\widehat{\mathcal A}(b)$ by construction. Suppose inductively that
all cluster variables in a seed $\mathbf t'$ obtained from
$\mathbf t_{\bfj}^{\,u}$ belong to $T_u\widehat{\mathcal A}(b)$. Let
$X_k$ be a mutable variable of $\mathbf t'$, and write the exchange
relation as
\[
X_k\,\mu_k(X_k)=q^A D_1+q^B D_2
\]
for some $A,B\in\mathbb Z$, where $D_1$ and $D_2$ are cluster monomials
in $\mathbf t'$.

Since $\mu_k(X_k)$ is a cluster variable in
$\mathcal A(\mathbf t_{\bfi})$, it is a global basis element. Write
\[
\mu_k(X_k)=B(\bfi,\bfa).
\]
By Lemma~\ref{lem_prodglobal}, the Lusztig parameter of the mutated
variable satisfies
\[
\bfa+\mathcal L_{\bfi}(X_k)
=
\max\{
\mathcal L_{\bfi}(D_1),
\mathcal L_{\bfi}(D_2)
\}.
\]
We use the following consequence of Lemma~\ref{lem:Tu-global-basis}: a
global basis element $B(\bfi,\bfa)$ of
$\widehat{\mathcal A}(\Delta_{\mathrm{Gar}}^m)$ lies in
$T_u\widehat{\mathcal A}(b)$ if
\[
\operatorname{supp}(\bfa)\subset \{n+1,\dots,n+\ell\}.
\]
By the induction hypothesis, the Lusztig parameters of all cluster
variables appearing in $D_1$ and $D_2$ are supported on the tail positions.
Hence the same is true for the cluster monomials $D_1$ and $D_2$.
The displayed formula then implies that $\bfa$ is also supported on the
tail positions. Therefore $B(\bfi,\bfa)\in T_u\widehat{\mathcal A}(b)$,
and so
\[
\mu_k(X_k)\in T_u\widehat{\mathcal A}(b).
\]
This completes the induction and proves
\[
\mathcal A(\mathbf t_{\bfj}^{\,u})
\subset
T_u\widehat{\mathcal A}(b).
\]

Conversely, by Theorem~\ref{thm:Mk-minors}, the mutation sequence $M_1$ for the word $\bfj$
produces the root vectors
\[
D_{\bfi}[s^-,s]=E_s^{\bfi},
\qquad
n+1\le s\le n+\ell,
\]
as cluster variables of $\mathcal A(\mathbf t_{\bfj}^{\,u})$. These root
vectors generate $T_u\widehat{\mathcal A}(b)$. Hence
\[
T_u\widehat{\mathcal A}(b)
\subset
\mathcal A(\mathbf t_{\bfj}^{\,u}).
\]
Therefore
\[
\mathcal A(\mathbf t_{\bfj}^{\,u})
=
T_u\widehat{\mathcal A}(b),
\]
and all cluster monomials of $\mathcal A(\mathbf t_{\bfj}^{\,u})$ belong
to the global basis.

Finally, by Lemma~\ref{lem:Tu-global-basis}, for every $a>n$ one has
\[
D_{\bfi}[a,n+\ell\}
=
T_u\bigl(D_{\bfj}[a-n,\ell\}\bigr).
\]
Moreover, under the shift $a\mapsto a-n$, the matrices
$\Lambda_{\bfj}$ and $B_{\bfj}$, together with the exchangeable set
$K^{\operatorname{ex}}$, agree with the corresponding tail restrictions
of $\Lambda_{\bfi}$, $B_{\bfi}$, and $K_{\bfi}^{\operatorname{ex}}$.
Thus $T_u$ identifies the seed $\mathbf t_{\bfj}$ with the tail seed
$\mathbf t_{\bfj}^{\,u}$.

Applying $T_u^{-1}$ to
\[
\mathcal A(\mathbf t_{\bfj}^{\,u})
=
T_u\widehat{\mathcal A}(b),
\]
we obtain
\[
\mathcal A(\mathbf t_{\bfj})
=
\widehat{\mathcal A}(b).
\]
Furthermore, Lemma~\ref{lem:Tu-global-basis} identifies the global basis
elements supported on the tail subword $\bfj$ with the corresponding
global basis elements of $\widehat{\mathcal A}(b)$. Hence all cluster
monomials of $\mathcal A(\mathbf t_{\bfj})$ belong to the global basis of
$\widehat{\mathcal A}(b)$. This proves the theorem.
\end{proof}

\subsection*{Declaration}
The author declares that he has no conflict of interest. The author used ChatGPT as an auxiliary tool to polish the exposition,
improve readability, and review the clarity of the proofs. All mathematical
arguments and conclusions were checked and are the responsibility of the
author.

\end{document}